\documentclass{article}
\usepackage[a4paper]{geometry}
\usepackage{amssymb}
\usepackage{systeme}
\usepackage[english]{babel}
\usepackage{color}
\usepackage{amsmath}
\usepackage{amsthm}
\usepackage{graphicx}
\usepackage{geometry}
\usepackage{mathtools}
\usepackage{comment}
\usepackage[normalem]{ulem}
\allowdisplaybreaks

\usepackage[urlcolor=black]{hyperref}
\hypersetup{
    colorlinks=true, 
    linktoc=all,     
    linkcolor=black,  
    citecolor=black
}

\theoremstyle{definition}
\newtheorem{theorem}{Theorem}[section]
\newtheorem{definition}[theorem]{Definition}

\newtheorem{corollary}[theorem]{Corollary}
\newtheorem{lemma}[theorem]{Lemma}
\newtheorem{remark}[theorem]{Remark}
\newtheorem{proposition}[theorem]{Proposition}

\title{Special Joyce structures and hyperk\"ahler metrics}
\date{\small School of Mathematics and Statistics\\
University of Sheffield\\
Hounsfield Road, Sheffield S3 7RH, United Kingdom}
\author{Iv\'an Tulli}
\numberwithin{equation}{section}

\begin{document}
\emergencystretch 3em
\maketitle
\begin{abstract}
    Joyce structures were introduced by T. Bridgeland in the context of the space of stability conditions of a three-dimensional Calabi-Yau category and its associated Donaldson-Thomas invariants. In subsequent work, T. Bridgeland and I. Strachan showed that Joyce structures satisfying a certain non-degeneracy condition encode a complex hyperk\"{a}hler structure on the tangent bundle of the base of the Joyce structure. In this work we give a definition of an analogous structure over an affine special K\"{a}hler (ASK) manifold, which we call a special Joyce structure. Furthermore, we show that it encodes a real hyperk\"{a}hler (HK) structure on the tangent bundle of the ASK manifold, possibly of indefinite signature. Particular examples include the semi-flat HK metric associated to an ASK manifold (also known as the rigid c-map metric) and the HK metrics associated to certain uncoupled variations of BPS structures over the ASK manifold. Finally, we relate the HK metrics coming from special Joyce structures to HK metrics on the total space of algebraic integrable systems. 
\end{abstract}
\tableofcontents
\section{Introduction}

The motivation for this work comes mainly from two sources. On one hand, the work of T. Bridgeland and I. Strachan  on Joyce structures and their associated complex hyperk\"ahler ($\mathbb{C}$-HK) structures \cite{BriStr}; and on the other hand, the work in the physics literature of D. Gaiotto, G. Moore and A. Neitzke, which associates a real hyperk\"ahler (HK) structure to a variation of BPS structures over an affine special K\"{a}hler (ASK) manifold \cite{GMN}. \\

Given a three-dimensional Calabi-Yau triangulated category $\mathcal{D}$, the notion of Joyce structure was introduced in \cite{BridgeJoyce} to describe a certain geometric structure on the space of stability conditions $\text{Stab}(\mathcal{D})$, encoded by the Donaldson-Thomas invariants of $\mathcal{D}$. Assuming certain non-degeneracy conditions are satisfied, the starting point of a Joyce structure is a holomorphic symplectic manifold $\mathcal{M}$ with a compatible flat and torsion-free connection. The data of a Joyce structure over $\mathcal{M}$ associates a family of flat holomorphic Ehresmann connections $\mathcal{A}^{\zeta}$ on $\pi: T\mathcal{M}\to \mathcal{M}$ parametrized by $\zeta\in \mathbb{C}^{\times}=\mathbb{C}\setminus\{0\}$, which must be symplectic and satisfy certain additional properties that we omit for the sake of brevity. Joyce structures are locally encoded in a single holomorphic function $J$ on an open subset of $T\mathcal{M}$, known as the Pleba\'nski function, which must satisfy a set of partial differential equations called Pleba\'nski's second heavenly equations (see \cite[Equation 1]{BriStr}). The way Joyce structures over $\text{Stab}(\mathcal{D})$ are built in the known examples (see \cite[Section 8, 9 and 10]{BridgeJoyce}, \cite{JoyceQuad} and \cite{A2Joyce}) is by solving a certain Riemann-Hilbert (RH) problem determined by the DT invariants and natural structures on $\text{Stab}(\mathcal{D})$ \cite{VarBPS}. The solutions of the RH problem are then used to define the corresponding family $\mathcal{A}^{\zeta}$ of flat holomorphic Ehresmann connections. In subsequent work by T. Bridgeland and I. Strachan \cite{BriStr}, they showed that a Joyce structure over $\mathcal{M}$ naturally encodes a $\mathbb{C}$-HK structure on (an open subset of) $T\mathcal{M}$.  Since Joyce structures are locally encoded by the Pleba\'nski function $J$, the same holds true for the associated $\mathbb{C}$-HK structure.\\

On the other hand, there is a strong parallel between the way Joyce structures are constructed and previous work in the physics literature by D. Gaiotto, G. Moore and A. Neitzke \cite{GMN}. In their work, they start with an ASK manifold $M$ together with a variation of BPS structures over $M$\footnote{We remark that they did not have the notion of variation of BPS structures, since this was introduced later in \cite{VarBPS}. Hence, that terminology does not appear in \cite{GMN}.}. Variations of BPS structures were introduced in \cite{VarBPS}, and can be thought as abstracting the natural structures present in $\text{Stab}(\mathcal{D})$ together with the associated DT invariants, or as abstracting natural structures appearing in the context of 4d $\mathcal{N}=2$ supersymmetric field theories (or the supergravity counterparts) and their BPS indices.  From this data they construct a real HK manifold on the cotangent bundle $T^*M$ by solving a certain RH problem which is related, but different, to the RH problem considered by T. Bridgeland. Physically, this HK metric is an instanton corrected metric of a moduli space associated to a 3d effective theory, obtained from the 4d $\mathcal{N}=2$ theory by compactifying on $S^1$. The construction of \cite{GMN} is mathematically well-understood when the variation of BPS structures is uncoupled\footnote{Also known as mutually-local.} (see for example \cite[Section 3]{CT} for a mathematical treatment of this case). The harder and not well-understood case of coupled variations of BPS structures is quite interesting and important, since particular cases are conjectured to give rise to the HK metrics of certain Higgs bundles moduli spaces \cite{GMNHitchin,NewCHK,Hitchin4dn2}. Some results hinting that this might be true have been obtained in the mathematics literature, for example \cite{FredricksonExpdecay,AsymptoticHitchin}. \\

In this work, we try to relate the above two perspectives by encoding a real HK geometry on the tangent bundle of an ASK manifold by something similar to a Joyce structure. Given an ASK manifold $M$, we define a structure  involving a $\mathbb{C}^{\times}$-family of complexified\footnote{A complexified Ehresmann connection is not in general the complex-linear extension of an Ehresmann connection. However, the latter gives examples of complexified Ehresmann connections. See Section \ref{complexEhressec}.} flat Ehresmann connections $\mathcal{A}^{\zeta}$ on $\pi:TM \to M$. Since such a structure is defined over a special K\"{a}hler manifold, we will call it a special Joyce structure. We nevertheless emphasize that these are not particular cases of the Joyce structures defined in \cite{BridgeJoyce}, but only similar structures. We show in Theorem \ref{maintheorem} that special Joyce structures encode a real HK structure on $TM$, possibly of indefinite signature. Analogously to Joyce structures, the family $\mathcal{A}^{\zeta}$ of a special Joyce structure is also determined by a  function $J$ on (an open subset of) $TM$. Contrary to usual Joyce structures, the function $J$ must now be smooth (instead of holomorphic), and must satisfy a more complicated set of partial differential equations \eqref{Plebanski-like} which nevertheless have similarities to the Pleba\'nski's second heavenly equations \eqref{Plebanskieq} appearing in \cite[Equation 1]{BriStr}. \\

We will see that particular examples of special Joyce structures recover the semi-flat HK metric \cite{SK,ACD}, and the HK metrics associated to uncoupled variations of BPS structures from \cite{GMN} discussed in detail in \cite{CT}. While these examples of HK metrics are more naturally defined in $T^*M$ instead of $TM$, we will see that the HK structure on $TM$ given by the special Joyce structure is related to the one in $T^*M$ by the natural identification $\omega: TM\cong T^*M$, $X \to \omega(X,-)$ given by the K\"{a}hler form $\omega$ of the ASK structure on $M$. Furthermore, even though in general the HK metrics that we get from special Joyce structures might be of indefinite signature, in the examples from above it is known how to control the signature of the resulting HK manifolds in terms of the signature of the ASK metric (see \cite[Section 3]{CT}).\\

The HK metrics constructed from special Joyce structures have a close connection to algebraic integrable systems. Indeed, when the ASK structure on $M$ is described by a period structure $(M,\Gamma,Z)$ (see Section \ref{specialperiodsec}), then provided the HK geometry on $TM$ induced by the special Joyce structure is invariant under fiberwise translations by $2\pi \cdot \Gamma\subset TM$\footnote{In the work of T. Bridgeland \cite{BridgeJoyce}, this kind of property is included in the definition of Joyce structure.}, we obtain an HK metric on $X:=TM/(2\pi \cdot\Gamma)$. We will then see that the canonical projection $\pi:X\to M$ has the structure of an algebraic integrable system compatible with the HK structure on $X$. The examples mentioned above of semi-flat metrics and HK metrics associated to uncoupled variations of BPS structures are instances of this.   \\

The particular form of the family $\mathcal{A}^{\zeta}$  appearing in the definition of special Joyce structures \eqref{famconn} is motivated by trying to generalize the family $\mathcal{A}^{\zeta}$ that occurs for HK metrics associated to uncoupled variations of BPS structures, discussed in Section \ref{uncoupled}. Nevertheless, it is currently unknown to the author whether special Joyce structures are able to capture the missing case of HK metrics associated to coupled variations of BPS structures. We comment on a possible strategy for checking this in Appendix \ref{instgenapp}\footnote{The author would like to thank S. Alexandrov and B. Pioline for suggesting this.}. If the HK metrics corresponding to coupled variations of BPS structures are not included in the HK metrics encoded by special Joyce structures, then this missing case is likely described by a structure that generalizes the current definition of special Joyce structures. In any case, the current definition  seems to be general enough to possibly allow other HK metrics beyond the examples discussed in Section \ref{examplesec}. For instance,  the function $J$ defining the family $\mathcal{A}^{\zeta}$ must in general satisfy a set of non-linear Pleba\'nski-like PDE's \eqref{Plebanski-like}, while in the particular case of uncoupled variation of BPS structures, the corresponding PDE's satisfied by $J$ simplify to linear PDE's \eqref{Plebanski-likelinear}. \\

Finally, we note that part of the definition of Joyce structures from T. Bridgeland requires the data of a holomorphic Euler vector field of $\mathcal{M}$. While in general we do not have the analog of this in the definition of special Joyce structure over $M$, what the analog should be is obvious when $M$ is conical affine special K\"{a}hler (CASK) instead of just ASK. In such a case $M$ comes equipped with a (real) Euler vector field, satisfying similar properties to the holomorphic Euler vector field from Joyce structures. Furthermore, in \cite{CT} it was shown that HK metrics associated to uncoupled variation of BPS structures over a CASK manifold admit an infinitesimal rotating action together with a hyperholomorphic line bundle\footnote{Here by infinitesimal rotating action we mean a Killing vector field which leaves invariant one of the K\"{a}hler forms of the HK structure, and rotates the other two. On the other hand, by hyperholomorphic line bundle we mean a line bundle with connection over the HK manifold, whose curvature is of type $(1,1)$ with respect to the three complex structures of the HK structure.}. One can then apply the HK/QK correspondence to obtain a quaternionic-K\"{a}hler (QK) manifold from the HK manifold. These QK metrics are related to QK structures of certain moduli spaces associated to Calabi-Yau compactifications of type IIA/B string theory. Whether a similar construction holds for HK manifolds associated to a special Joyce structure over a CASK base will be deferred to future work. 

\subsection{Structure of the paper}
\begin{itemize}
    \item In Section \ref{preliminariessec} we start by recalling some well-known facts about affine special K\"ahler (ASK) manifolds and some particular sets of coordinates adapted to the ASK structure. Everything we say about ASK manifolds is either contained in \cite{SK} or \cite{ACD}. We also recall the notion of Ehresmann connection and a slight extension of it, called complexified Ehresmann connection, and give some facts that will be useful for later.   
    \item In Section \ref{SpecialJoycesec} we start discussing certain families $\mathcal{A}^{\zeta}$ of complexified Ehresmann connections parametrized by $\zeta\in\mathbb{C}^{\times}$ and their relation to hypercomplex structures. We then specialize the previous family $\mathcal{A}^{\zeta}$ , and discuss the notion of almost special Joyce structure. When the family $\mathcal{A}^{\zeta}$ associated to the almost special Joyce structure is flat for each $
    \zeta \in \mathbb{C}^{\times}$ and satisfies a certain compatibility property with the ASK structure on $M$, we obtain the notion of special Joyce structure (see Definition \ref{specialJoycedef}). Given a special Joyce structure over $M$, we show that one can associate a natural hyperhermitian structure on $TM$. The main theorem (see Theorem \ref{maintheorem}) shows that this hyperhermitian structure is actually hyperk\"ahler.

    \item In Section \ref{examplesec} we discuss two classes of examples. The first corresponds to the trivial special Joyce structure over an ASK manifold, which recovers the semi-flat HK metric associated to the ASK manifold. The second class is a lot more interesting, and recovers HK manifolds associated to uncoupled variations of BPS structures over an ASK manifold. In particular, we write a function $J$ specifying the special Joyce structure explicitly (see \eqref{Juncoupled}), and check that the required PDE's \eqref{Plebanski-like} are satisfied. For the convenience of the reader, we recall the notion of special period structure and variations of BPS structures, required to understand the example. 
    \item In Section \ref{intsysrel} we discuss the relation between special Joyce structures and algebraic integrable systems. Namely, we consider the case where the ASK manifold comes from a special period structure $(M,\Gamma,Z)$ (see Definition \ref{SPSdef}) and the special Joyce structure is compatible with the period structure. In such a case, one obtains an HK structure on $X=TM/(2\pi \cdot \Gamma)$, and the HK structure induces on the canonical projection $X\to M$ the structure of an algebraic integrable system. All our examples from Section \ref{examplesec} will be instances of this. 

\end{itemize}
\subsection{Conventions}

Unless otherwise stated, all objects and morphisms are smooth. We will frequently disregard the signature of pseudo-Riemannian metrics and refer to pseudo-K\"{a}hler, pseudo-hermitian or pseudo-hyperk\"{a}hler manifolds simply as K\"{a}hler, hermitian or hyperk\"{a}hler. All Hamiltonian vector fields and Poisson brackets are with respect to the symplectic structure on the vertical bundle induced by the special K\"{a}hler form (see Section \ref{vertsympstructure}). Einstein summation convention is used throughout the paper. \\

\textbf{Acknowledgements:} the author is very grateful to S. Alexandrov, T. Bridgeland, A. Neitzke and B. Pioline for very useful discussions and comments. The author would furthermore like to thank V. Cort\'es, Alejandro Gil-Garc\'ia and A. Saha for discussions about a related work in progress \cite{specialHK}; and M. Alim, V. Cort\'es, J. Teschner and T. Weigand for discussions concerning an earlier version of the ideas presented here. 

\section{Preliminaries}\label{preliminariessec}
In this section we discuss some well-known facts about affine special K\"{a}hler (ASK) manifolds, and some basic facts about Ehresmann connections and the related notion of complexified Ehresmann connections. All the results about ASK manifolds can be found in \cite{SK,ACD}, but we include some proofs in order to be as self-contained as possible and to fix notations. 
\subsection{Affine special K\"{a}hler manifolds}

\begin{definition}
    An affine special K\"{a}hler (ASK) manifold is a tuple $(M,I,\omega,\nabla)$ such that:
    \begin{itemize}
        \item $(M,I,\omega)$ is a pseudo-K\"{a}hler manifold, where $I$ is the complex structure and $\omega$ the K\"{a}hler form. The possibly indefinite metric is given by $g(-,-)=\omega(-,I-)$.
        \item $\nabla$ is a flat, torsion-free connection on $M$.
        \item $\nabla \omega=0$ and $\mathrm{d}_{\nabla}I=0$, where in the latter we think of $I$ as an element of $\Omega^{1}(M,TM)$ (i.e. 1-form with values in $TM$) and $d_{\nabla}$ is the natural extension of $\nabla:\Omega^0(M,TM)\to \Omega^1(M,TM)$ to higher degree forms $d_{\nabla}:\Omega^k(M,TM)\to \Omega^{k+1}(M,TM)$.
    \end{itemize}
\end{definition}
We now recall some well-known results from \cite{SK,ACD}.
\begin{lemma}\label{affinespecialcoords} Given an ASK manifold $(M,I,\omega,\nabla)$, there exist locally flat Darboux coordinates $(x^i,y_i)$ for $\omega$. That is 
\begin{equation}
    \omega=\mathrm{d}x^i\wedge \mathrm{d}y_i\,, \quad \nabla (\mathrm{d}x^i)=\nabla (\mathrm{d}y_i)=0\,.
\end{equation}
\end{lemma}
\begin{proof}
Let $v^i$ be a local flat frame of $T^*M$ with respect to $\nabla$ around $p\in M$. There is a constant linear change of frame $(v^i)_{i=1,...,2\text{dim}_{\mathbb{C}}(M)} \to (\xi^i,\xi_i)_{i=1,...,\text{dim}_{\mathbb{C}}(M)}$ such that at $p\in M$
\begin{equation}\label{pointDarboux}
    \omega|_p=\xi^i\wedge\xi_i|_p\,.
\end{equation}
Since $\nabla \omega=0$ and $(\xi^i,\xi_i)$ are flat, \eqref{pointDarboux} holds on an open set containing $p$. On the other hand, the torsion-free condition can be written as
\begin{equation}\label{torfree}
    \mathrm{d}_{\nabla}(\text{Id}_{TM})=0\,,
\end{equation}
where $\text{Id}_{TM}\in \Omega^1(M,MT)$ is the identity map on $TM$.
Letting $(\eta_i,\eta^i)$ be the flat frame of $TM$ dual to $(\xi^i,\xi_i)$, we conclude from \eqref{torfree} and $\text{Id}_{TM}=\xi^i\otimes \eta_i + \xi_i\otimes \eta^i$ that 
\begin{equation}
    \mathrm{d}\xi^i\otimes \eta_i + \mathrm{d}\xi_i\otimes \eta^i=0\,.
\end{equation}
It then follows that $\mathrm{d}\xi^i=\mathrm{d}\xi_i=0$, and hence locally there is a coordinate system $(x^i,y_i)$ such that
\begin{equation}
    \mathrm{d}x^i=\xi^i, \quad \mathrm{d}y_i=\xi_i\,.
\end{equation}
The result then follows.
\end{proof}

\begin{definition}
    Coordinates $(x^i,y_i)$ on an ASK manifold as in Lemma \ref{affinespecialcoords} are called affine special coordinates. 
\end{definition}
\begin{lemma}\cite[Theorem 1]{ACD}\label{coordlemma} Given an ASK manifold $(M,I,\omega,\nabla)$, locally around any point $p\in M$ there are two associated systems of holomorphic coordinates $(Z^i)_{i=1,...,\text{dim}_{\mathbb{C}}(M)}$ and $(Z_i)_{i=1,...,\text{dim}_{\mathbb{C}}(M)}$ for $(M,I)$ such that $\text{Re}(Z^i)=x^i$, $\text{Re}(Z_i)=-y_i$ define affine special coordinates.
\end{lemma}
\begin{proof}
    Consider the projection $\pi^{1,0}:TM\otimes \mathbb{C}\to T^{1,0}M$ into $(1,0)$ vectors with respect to $I$ given by
    \begin{equation}
        \pi^{1,0}=\frac{1}{2}(\mathrm{Id}_{TM}-\mathrm{i}I)\,.
    \end{equation}
    It can be thought as an element of $\Omega^{1,0}(M,TM\otimes \mathbb{C})$.
    The fact that $\mathrm{d}_{\nabla}(I)=0$ and the torsion freeness of $\nabla$ \eqref{torfree} imply that 
    \begin{equation}
        \mathrm{d}_{\nabla} \pi^{1,0}=0\,.
    \end{equation}
    By the Poincar\'e lemma, there is locally a complex vector field $\xi^{1,0}$ such that
    \begin{equation}
        \nabla \xi^{1,0}=\pi^{1,0}\,.
    \end{equation}
    Let $(\gamma_i,\gamma^i)$ be a local flat Darboux frame of $TM$ with respect to $\omega$ around $p\in M$. Then  
    \begin{equation}\label{conjugateholspecial}
        \xi^{1,0}=\frac{1}{2}(Z^i\gamma_i-Z_i\gamma^i)
    \end{equation}
    for some locally defined complex-valued functions $Z^i$ and $Z_i$ on $M$. But then
    \begin{equation}
        \pi^{1,0}=\nabla \xi^{1,0}=\frac{1}{2}(\mathrm{d}Z^i\otimes \gamma_i - \mathrm{d}Z_i\otimes \gamma^i) \in \Omega^{1,0}(M,TM\otimes \mathbb{C})
    \end{equation}
    implies that $\mathrm{d}Z^i, \mathrm{d}Z^i\in \Omega^{1,0}(M)$, and hence $Z^i$ and $Z_i$ are holomorphic functions on $(M,I)$. Note that 
    \begin{equation}
        2\mathrm{Re}(\pi^{1,0})=\text{Id}_{TM}
    \end{equation}
    implies that $(x^i=\text{Re}(Z^i),y_i=-\text{Re}(Z_i))$ is a local coordinate system around $p\in M$ and that $\gamma_i=\partial_{x^i}$, $\gamma^i=\partial_{y_i}$. Hence $(x^i,y_i)$ is flat and Darboux with respect to $\omega$.\\
    
    We now show that both sets of holomorphic functions $(Z^i)$ and $(Z_i)$ define coordinates systems on $M$. To see this, note that $(x^i,y_i)$ defines a Lagrangian splitting of $T^*M$ (with respect to $\omega^{-1}$)
    \begin{equation}
        T^*_pM=L_x\oplus L_y, \quad L_x:=\text{span}\{\mathrm{d}x^i\}_{i=1,...,\text{dim}_{\mathbb{C}}(M)}, \quad L_y:=\text{span}\{\mathrm{d}y_i\}_{i=1,...,\text{dim}_{\mathbb{C}}(M)}\,.
    \end{equation}
    Due to the compatibility of $\omega$ and $I$, by possibly performing a constant symplectic linear change of coordinates, we can assume that 
    \begin{equation}
        L_x\cap I^*L_x=\{0\}, \quad L_y\cap I^*L_y=\{0\}.
    \end{equation}
    
    Noting that 
    \begin{equation}
        \mathrm{d}Z^i=\mathrm{d}x^i-\mathrm{i}I^*\mathrm{d}x^i, \quad \mathrm{d}Z_i=-\mathrm{d}y_i+\mathrm{i}I^*\mathrm{d}y_i\,,
    \end{equation}
    it then follows from the independence of the $\mathrm{d}x^i$ (resp. $\mathrm{d}y_i$) that the $\mathrm{d}Z^i$ (resp. $\mathrm{d}Z_i$) are independent at $p\in M$, and hence locally around $p\in M$. It follows that $(Z^i)$ and $(Z_i)$ are holomorphic coordinate systems of $M$.
\end{proof}
\begin{definition}
    Two systems of holomorphic coordinates $(Z^i)$ and $(Z_i)$ on an ASK manifold obtained as in \eqref{conjugateholspecial} are called a conjugate systems of holomorphic special coordinates. 
\end{definition}

\begin{lemma}\label{usualsKident} Consider a conjugate system of holomorphic special coordinates $(Z^i)$ and $(Z_i)$ for an ASK manifold. 
Define $\tau_{ij}$ by
\begin{equation}
    \mathrm{d}Z_i=\tau_{ij}\mathrm{d}Z^i\,.
\end{equation}
Then there exists a local holomorphic function $\mathfrak{F}(Z^i)$ such that
\begin{equation}
    \tau_{ij}=\frac{\partial^2 \mathfrak{F}}{\partial Z^i \partial Z^j}. 
\end{equation}
In particular, we must have $\tau_{ij}=\tau_{ji}$, and with respect to the holomorphic coordinates $(Z^i)$ we have 
\begin{equation}
    \omega=\frac{\mathrm{i}}{2}\text{Im}(\tau_{ij})\mathrm{d}Z^i\wedge \mathrm{d}\overline{Z}^j\,.
\end{equation}
\end{lemma}
\begin{proof}
    Note that writing as in the proof of Lemma \ref{coordlemma} 
    \begin{equation}
        \xi^{1,0}=\frac{1}{2}\left(Z^i\frac{\partial}{\partial x^i}-Z_i\frac{\partial}{\partial y_i}\right)
    \end{equation}
    we find 
    \begin{equation}
        \pi^{1,0}=\frac{1}{2}\left(\mathrm{d}Z^i\otimes \frac{\partial}{\partial x^i} -\tau_{ij}\mathrm{d}Z^j\otimes \frac{\partial}{\partial y_i}\right)\,. 
    \end{equation}
    Evaluating the above at $\partial_{Z^i}$ one finds
    \begin{equation}\label{Zexp}
        \partial_{Z^i}=\frac{1}{2}\left(\frac{\partial}{\partial x^i}-\tau_{ji}\frac{\partial}{\partial y_i}\right)\,.
    \end{equation}
    On the other hand, using that $\omega$ is of type $(1,1)$, we find
    \begin{equation}
        0=\omega\left(\frac{\partial}{\partial Z^i},\frac{\partial}{\partial Z^j}\right)=\frac{1}{4}(\tau_{ji}-\tau_{ij})
    \end{equation}
    so that $\tau_{ij}=\tau_{ji}$. It follows that the holomorphic $1$-form
    \begin{equation}
        Z_i\mathrm{d}Z^i
    \end{equation}
    is closed, and hence locally there is a holomorphic function $\mathfrak{F}$ such that
    \begin{equation}
        \mathrm{d}\mathfrak{F}=Z_i\mathrm{d}Z^i, \quad \implies Z_i=\frac{\partial \mathfrak{F}}{\partial Z^i}, \quad  \tau_{ij}=\frac{\partial^2 \mathfrak{F}}{\partial Z^i \partial Z^j}\,.
    \end{equation}
    Finally, note that
    \begin{equation}
        \omega(\frac{\partial}{\partial Z^i}, \frac{\partial}{\partial \overline{Z}^j})=\frac{1}{4}(\tau_{ji}-\overline{\tau}_{ij})=\frac{\mathrm{i}}{2}\text{Im}(\tau_{ij})
    \end{equation}
    so that 
    \begin{equation}
        \omega=\frac{\mathrm{i}}{2}\text{Im}(\tau_{ij})\mathrm{d}Z^i\wedge \mathrm{d}\overline{Z}^j\,.
    \end{equation}
\end{proof}
\subsection{Complexified Ehresmann connections}\label{complexEhressec}

Consider a smooth submersion $\pi:N\to M$ and let $V_{\pi}:=\text{Ker}(\mathrm{d}\pi)\subset TN$ denote the vertical bundle associated to $\pi$. This gives rise to the short exact sequence of vector bundles over $N$
\begin{equation}
   0\longrightarrow V_{\pi}\overset{i}{\longrightarrow } TN \overset{\mathrm{d}\pi}{\longrightarrow } \pi^*(TM)\longrightarrow 0\,. 
\end{equation}
\begin{definition}
An Ehresmann connection on $\pi:N\to M$ is a splitting of the above short exact sequence. That is, a vector bundle map $\mathcal{A}:\pi^*(TM)\to TN$ such that $\mathrm{d}\pi\circ \mathcal{A}=1$. Given $X\in \pi^*(TM)$ we use the notation $\mathcal{A}_X:=\mathcal{A}(X)$ for the evaluation.  
\end{definition}
\begin{definition}\label{complexEhr}
    Let $\pi:N \to M$ as before and consider the complexified short exact sequence
    \begin{equation}
   0\longrightarrow V_{\pi}\otimes \mathbb{C}\overset{i}{\longrightarrow } TN\otimes \mathbb{C} \overset{\mathrm{d}\pi}{\longrightarrow } \pi^*(TM)\otimes \mathbb{C}\longrightarrow 0\,.
\end{equation}
A complexified Ehresmann connection on $\pi:N\to M$ is a complex vector bundle map $\mathcal{A}:\pi^*(TM)\otimes \mathbb{C}\to TN\otimes \mathbb{C}$ such that $\mathrm{d}\pi\circ \mathcal{A}=1$
\end{definition}

\begin{definition}
An Ehresmann connection (resp. complexified Ehresmann connection) $\mathcal{A}$ is flat if the distribution $\text{Im}(\mathcal{A})\subset TN$ (resp. $\text{Im}(\mathcal{A})\subset TN\otimes \mathbb{C})$ is involutive. Namely, given any local sections of $X,Y$ of $\text{Im}(\mathcal{A})\to N$, $[X,Y]$ is also a local section of $\text{Im}(\mathcal{A})\to N$.
\end{definition}

\begin{remark}\label{pullbackremark}
    We frequently abuse notation and evaluate $\mathcal{A}$ on local sections of $TM\to M$ (resp. $TM\otimes \mathbb{C}\to M$), with the understanding that we evaluate it on the canonically induced local section $X\circ \pi$ of $\pi^*TM\to N$ (resp. $\pi^*TM\otimes \mathbb{C} \to N$). Note in particular that since $\pi^*(TM)\to N$ admits local frames of such pullback sections, in order to check flatness it is enough to check that for every local frame  $(e_i)$ of $TM\to M$, we have
    \begin{equation}
        [\mathcal{A}_{e_i},\mathcal{A}_{e_j}]\subset \text{span}\{\mathcal{A}_{e_i}\}_{i=1,..,\text{dim}(M)}\,.
    \end{equation}
\end{remark}
We will also make frequent use of the following lemma
\begin{lemma}\label{Ehrflatlemma}
    Assume that $N=TM$, and $\pi:TM\to M$ is the canonical projection. Then an Ehresmann connection $\mathcal{A}$ (resp. complexified Ehresmann connection) on $\pi:TM\to M$ is flat if and only if for all local sections $X,Y$ of $TM \to M$ (resp. $TM\otimes \mathbb{C}\to M$) we have 
    \begin{equation}\label{flatsimp}
[\mathcal{A}_{X},\mathcal{A}_Y]=\mathcal{A}_{[X,Y]}\,.
    \end{equation}
\end{lemma}
\begin{proof}
    We do the proof for an Ehresmann connection, with the proof for a complexified Ehresmann connection following identically. If \eqref{flatsimp} holds, by Remark \ref{pullbackremark} the connection is flat. On the other hand, assume that the connection is flat and take local coordinates $x^i$ on $M$. They induce coordinates $(x^i,\varphi^i)$ on $TM$. By the definition of an Ehresmann connection, we must have (with the usual abuse of notation from Remark \ref{pullbackremark})
    \begin{equation}
        \mathcal{A}_{\frac{\partial}{\partial x^i}}=\frac{\partial}{\partial x^i}+f^k_i\frac{\partial}{\partial \varphi^k}
    \end{equation}
    for some functions $f^k_i$ on $TM$.
    In particular, 
    \begin{equation}
        [\mathcal{A}_{\frac{\partial}{\partial x^i}},\mathcal{A}_{\frac{\partial}{\partial x^j}}]=\left(\frac{\partial f_j^k}{\partial x^i}-\frac{\partial f_i^k}{\partial x^j}\right)\frac{\partial}{\partial \varphi^k}\,.
    \end{equation}
    But since the latter is a section of $V_{\pi}$,
    the flatness of $\mathcal{A}$ implies that the above quantity must be zero, and hence 

     \begin{equation}
        [\mathcal{A}_{\frac{\partial}{\partial x^i}},\mathcal{A}_{\frac{\partial}{\partial x^j}}]=0\,.
    \end{equation}
    Now note that if $X,Y$ are local sections of $TM \to M$, then with respect to the local coordinate system $(x^i,\varphi^i)$ from before
    \begin{equation}
    \begin{split}
[\mathcal{A}_X,\mathcal{A}_Y]&=X^i\mathcal{A}_{\frac{\partial}{\partial x^i}}(Y^j)\mathcal{A}_{\frac{\partial}{\partial x^j}}-Y^i\mathcal{A}_{\frac{\partial}{\partial x^i}}(X^j)\mathcal{A}_{\frac{\partial}{\partial x^j}}+X^iY^j[\mathcal{A}_{\frac{\partial}{\partial x^i}},\mathcal{A}_{\frac{\partial}{\partial x^j}}]\\
&=X^i\frac{\partial}{\partial x^i}(Y^j)\mathcal{A}_{\frac{\partial}{\partial x^j}}-Y^i\frac{\partial}{\partial x^i}(X^j)\mathcal{A}_{\frac{\partial}{\partial x^j}}\\
&=\mathcal{A}_{[X,Y]}\,.
\end{split}
    \end{equation}
Were we used that the functions $X^i$ and $Y^j$ depend only on $x^i$ (since $X$ and $Y$ are local section of $TM\to M$). 
\end{proof}

For future reference, we note that a connection $\nabla$ on $M$ induces a natural Ehresmann connection $\mathcal{A}$ on $\pi:TM \to M$ as follows. Let $V\in T_pM$ and $X\in T_pM$, so that $(V,X)\in \pi^*(TM)$. Furthermore, let $\gamma$ be a curve in $M$ such that $\gamma(0)=p$ and $\dot{\gamma}(0)=X$, and let 
\begin{equation}
    P_{\gamma,t}:T_{\gamma(0)}M\to T_{\gamma(t)}M
\end{equation}
denote the parallel transport induced by $\nabla$. We then obtain the curve $t\to P_{\gamma,t}(V)$ in $TM$ and define
\begin{equation}\label{inducedconndef}
    \mathcal{A}_{(V,X)}:=\frac{d}{dt}P_{\gamma,t}(V)\Big|_{t=0}\in T_V(TM)\,.
\end{equation}
It is easy to check that $\mathcal{A}_{(V,X)}$ is well defined (i.e. it does not depend on the curve $\gamma$ chosen such that $\gamma(0)=p$ and $\dot{\gamma}(0)=X$). Indeed if $(x^i)$ are local coordinates on $M$ and $(x^i,\varphi^i)$ the induced coordinates on $TM$, then one easily checks that 
\begin{equation}\label{inducedconn}
    \mathcal{A}_{(V,X)}=\frac{d}{dt}P_{\gamma,t}(V)\Big|_{t=0}=X^i\left(\frac{\partial}{\partial x^i} -V^j\Gamma_{ij}^k\frac{\partial}{\partial \varphi^k}\right), \quad \text{where}\quad  X=X^i\frac{\partial}{\partial x^i}, \quad V=V^i\frac{\partial}{\partial x^i}\,,
\end{equation}
and $\Gamma_{ij}^k$ are the Christoffel symbols of $\nabla$ with respect to the coordinates $(x^i)$. From \eqref{inducedconn} it immediately follows that $\mathcal{A}$ as defined in \eqref{inducedconndef} gives an Ehresmann connection on $\pi: TM \to M$.

\section{Special Joyce structures and hyperk\"{a}hler structures}\label{SpecialJoycesec}

In this section we study certain families  $\mathcal{A}^{\zeta}$ of complexified Ehresmann connections on $\pi:TM \to M$, and associated geometric structures on $TM$. We start with a rather general family $\mathcal{A}^{\zeta}$ related to hypercomplex structures on $TM$, and then focus on a more particular family to define special Joyce structures. A special Joyce structure has a natural hyperhermitian structure defined in terms of the ASK structure and the family $\mathcal{A}^{\zeta}$, and the main theorem (see Theorem \ref{maintheorem}) states that the hyperhermitian structure is actually hyperk\"{a}hler.  

\subsection{Complexified Ehresmann connections and hypercomplex structures}\label{hypercomplexsection}

In this section we take a complex manifold $(M,I)$ and take $N:=TM$ with the natural projection $\pi:N\to M$. We furthermore denote by $TM\otimes \mathbb{C}=T^{1,0}M\oplus T^{0,1}M$ the usual splitting given by the eigenspaces of $I$. We consider a complex vector bundle map
\begin{equation}\label{hdef}
    h:\pi^*(T^{1,0}M)\to TN\otimes \mathbb{C}\,,
\end{equation}
together with an antilinear\footnote{The reason for taking $v$ complex anti-linear instead of complex linear is only a matter of convention. The convention is such that in complex structure $I_3$ on $N$ (to be defined later), the $(1,0)$ vectors are spanned by $h_X$ and $v_X$ for $X\in T^{1,0}M$, instead of $h_X$ and $\overline{v_X}$ for $X\in T^{1,0}M$.} complex vector bundle map 
\begin{equation}\label{vdef}
    v:\pi^*(T^{1,0}M)\to V_{\pi}\otimes \mathbb{C}=\text{Ker}(\mathrm{d}\pi)\otimes \mathbb{C}\,.
\end{equation}
That is, for $X\in \pi^*(T^{1,0}M)$ and $\lambda \in \mathbb{C}$ we have
\begin{equation}
    v_{\lambda X}=\overline{\lambda}v_X\,.
\end{equation}
We further assume the non-degeneracy condition 
\begin{equation}\label{non-deg}
    TN\otimes \mathbb{C}=\text{Im}(h)\oplus \text{Im}(\overline{h})\oplus \text{Im}(v)\oplus \text{Im}(\overline{v})\,, \; \quad d\pi \circ h_X=X, \quad d\pi\circ \overline{h_X}=\overline{X}\,.
\end{equation}
\begin{remark}\label{hvrealremark}
    If we wish, we can extend the definition of $h$, and for $X\in \pi^*(T^{1,0}M)$ let, 
    \begin{equation}
h_{\overline{X}}:=\overline{h_X}\,.
    \end{equation}
    It then follows that $h$ is the complex linear extension of a real vector bundle map
    \begin{equation}
        h:\pi^*(TM)\to TN\,.
    \end{equation}
    The last two conditions in \eqref{non-deg} then say that $h$ defines an Ehresmann connection on $\pi:TM \to M$. A similar extension can be done for the definition of $v$, so that it comes from the complex anti-linear extension of a real vector bundle map
    \begin{equation}
        v: \pi^*(TM)\to V_{\pi}\,.
    \end{equation}
    Nevertheless, we will continue think of $h$ and $v$ as in \eqref{hdef} and \eqref{vdef}.
\end{remark}
From this data, we want to consider a family of complexified Ehresmann connections $\mathcal{A}^{\zeta}$ on $\pi:N\to M$ parametrized by $\zeta\in \mathbb{C}^{\times}:=\mathbb{C}\setminus\{0\}$.  We assume that $\mathcal{A}^{\zeta}$ is given as follows for $X\in\pi^*(T^{1,0}M)$ 
\begin{equation}\label{Ehrcongen}
\begin{split}
\mathcal{A}_{X}^{\zeta}&=h_X-\frac{1}{\zeta}\overline{v_X}\\
\mathcal{A}_{\overline{X}}^{\zeta}&=\overline{h_X}+\zeta v_X\,.
\end{split}
\end{equation}
By $\eqref{non-deg}$, it immediately follows that $\mathcal{A}^{\zeta}$ is a complexified Ehresmann connection. However, note that we always have 
\begin{equation}
    \mathcal{A}_{X}^{\zeta}\neq \overline{\mathcal{A}_{\overline{X}}^{\zeta}}\,
\end{equation}
so $\mathcal{A}^{\zeta}$ is not the complex linear extension of an Ehresmann connection. \\

We also let for $\zeta=0,\infty$ and $X\in \pi^*(T^{1,0}M)$
\begin{equation}\label{conext}
\begin{split}
\mathcal{A}^{\zeta=0}_X&:=\zeta\mathcal{A}_X^{\zeta}|_{\zeta=0}=-\overline{v_X}, \quad\quad  \mathcal{A}^{\zeta=0}_{\overline{X}}:=\mathcal{A}_{\overline{X}}^{\zeta}|_{\zeta=0}=\overline{h_X}\\
\mathcal{A}^{\zeta=\infty}_X&:=\mathcal{A}_X^{\zeta}|_{\zeta=\infty}=h_X, \quad\quad  \mathcal{A}^{\zeta=\infty}_{\overline{X}}:=\frac{1}{\zeta}\mathcal{A}_{\overline{X}}^{\zeta}|_{\zeta=\infty}=v_X\\
\end{split}
\end{equation}
Note that $\mathcal{A}^{\zeta=0}$ and $\mathcal{A}^{\zeta=\infty}$ are not complexified Ehresmann connections on $\pi:TM \to M$.\\

For future reference, we reformulate the flatness condition of $\mathcal{A}^{\zeta}$ for each $\zeta \in \mathbb{C}^{\times}$ in the following lemma.
\begin{lemma}\label{flatlemma}
    The family of complexified Ehresmann connections $\mathcal{A}^{\zeta}$ defined in \eqref{Ehrcongen} is flat for all $\zeta\in \mathbb{C}^{\times}$ if and only if the following equations hold for all $X,Y$ local holomorphic sections of $T^{1,0}M\to M$:
    \begin{equation}\label{flateq}
        \begin{split}
         [h_X,h_Y]&=h_{[X,Y]}, \quad [h_X,\overline{v_Y}]+[\overline{v_X},h_Y]=\overline{v_{[X,Y]}}, \quad [v_X, v_Y]=0\,\\
        [h_X,\overline{h_Y}]&=[\overline{v_X},v_Y], \quad [h_X,v_Y]=0\,.\\
        \end{split}
    \end{equation}
\end{lemma}
\begin{proof}
    We make use of Lemma \ref{Ehrflatlemma}. From the condition $[\mathcal{A}_X,\mathcal{A}_Y]=\mathcal{A}_{[X,Y]}$ and the definition \eqref{Ehrcongen} we obtain the constraints 
\begin{equation}
    [h_X,h_Y]=h_{[X,Y]}, \quad [h_X,\overline{v_Y}]+[\overline{v_X},h_Y]=\overline{v_{[X,Y]}}, \quad [\overline{v_X}, \overline{v_Y}]=0\,.
\end{equation}
With similar constraints from $[\mathcal{A}_{\overline{X}},\mathcal{A}_{\overline{Y}}]=\mathcal{A}_{[\overline{X},\overline{Y}]}$, namely the conjugate from the above
\begin{equation}
    [\overline{h_X},\overline{h_Y}]=\overline{h_{[X,Y]}}, \quad [\overline{h_X},v_Y]+[v_X,\overline{h_Y}]=v_{[X,Y]}, \quad [v_X, v_Y]=0\,.
\end{equation}
Finally, since $X$ and $Y$ are holomorphic, we have $[X,\overline{Y}]=0$. Hence, from $[\mathcal{A}_X,\mathcal{A}_{\overline{Y}}]=\mathcal{A}_{[X,{\overline{Y}}]}=0$ we obtain 
\begin{equation}
    [h_X,\overline{h_Y}]=[\overline{v_X},v_Y], \quad [h_X,v_Y]=0, \quad [\overline{v_X},\overline{h_Y}]=0\,.
\end{equation}
\end{proof}
\begin{remark}
    Recall from Remark \ref{hvrealremark} that we can think of $h$ as coming from an Ehresmann connection on $TM\to M$, and hence as a complexified Ehresmann connection by extending complex-linearly. Note that while the first equation in \eqref{flateq} is a flatness condition when $X$ and $Y$ are holomorphic sections of $T^{1,0}M\to M$, we only have $[h_X,\overline{h_Y}]=[\overline{v_X},v_Y]$, so flatness of $h$ as a complexified Ehresmann connection is not guaranteed. In fact, in the main example of Section \ref{uncoupled} one can check that $[h_X,\overline{h_Y}]\neq 0=h_{[X,\overline{Y}]}$, so in that case $h$ is non-flat.
\end{remark}
\subsubsection{The associated hypercomplex structure}

Now note that for $\zeta \in \mathbb{C}^{\times}$ or $\zeta=0,\infty$ it follows from the definitions \eqref{Ehrcongen}, \eqref{conext}, together with \eqref{non-deg} that
\begin{equation}
    TN\otimes \mathbb{C}=\text{Im}(\mathcal{A}^{\zeta})\oplus\text{Im}(\overline{\mathcal{A}^{\zeta}})\,.
\end{equation}
Hence, we can define almost complex structures $I_{\zeta}$ on $N$ by letting $\text{Im}(A^{\zeta})$ be the $\mathrm{-i}$-eigenspace and $\text{Im}(\overline{\mathcal{A}^{\zeta}})$ the $\mathrm{i}$-eigenspace of $I_{\zeta}$. We furthermore define $I_i$ for $i=1,2,3$ to be $I_{\zeta}$ for $\zeta=\mathrm{i},-1,0$ respectively. 

\begin{proposition}\label{almosthypercomplex} The almost complex structures $I_i$ on $N$ satisfy the quaternion relations 
\begin{equation}
    I_1I_2=I_3, \quad I_iI_j=-I_jI_i \quad \text{for} \quad i\neq j\,.
\end{equation}
In particular, $(N,I_1,I_2,I_3)$ is an almost hypercomplex manifold. 
\begin{proof}
    From the definitions of $I_i$ it easily follows that 
        \begin{equation}\label{cstr}
    \begin{split}
        I_3(h_X)&=\mathrm{i}h_X, \quad I_3(v_X)=\mathrm{i}v_X\\
        I_1(h_X)&=\overline{v_X}, \quad I_1(v_X)=-\overline{h_X}\\
        I_2(h_X)&=-\mathrm{i}\overline{v_X}, \quad I_2(v_X)=\mathrm{i}\overline{h_X}\,,
    \end{split}
    \end{equation}
with the other evaluations at $\overline{h_X}$ and $\overline{v_X}$ determined by the fact that the $I_i$ are real operators. From the above the required quaternion relations follow. 
\end{proof}
\end{proposition}
We remark that we can express the ``twistor" family of almost complex structures $I_{\zeta}$ for $\zeta\in \mathbb{C}\subset \mathbb{C}P^1$ in terms of the $I_i$ via the stereographic projection formula

\begin{equation}\label{twistorhol}
    I_{\zeta}=\frac{\mathrm{i}(-\zeta+\overline{\zeta})I_1-(\zeta+\overline{\zeta})I_2+(1-|\zeta|^2)I_3}{1+|\zeta|^2}\,.
\end{equation}

\begin{corollary}\label{hypercomplexprop}
    If $\mathcal{A}^{\zeta}$ is flat for all $\zeta\in \mathbb{C}^{\times}$ then the almost complex structures $I_j$ are integrable. In particular, $(N,I_1,I_2,I_3)$ is a hypercomplex manifold.
\end{corollary}
\begin{proof}
   The flatness of $\mathcal{A}^{\zeta}$ for all $\zeta \in \mathbb{C}^{\times}$ implies that the distribution of $-\mathrm{i}$-eigenspaces of $I_{\zeta}$ is involutive for $\zeta\in \mathbb{C}^{\times}$, so in particular $I_1=I_{\zeta=\mathrm{i}}$ and $I_2=I_{\zeta=-1}$ are integrable. On the other hand, by Lemma \ref{flatlemma} (in particular \eqref{flateq}) it follows that $I_3=I_{\zeta=0}$ is also integrable. 
    
\end{proof}

\subsection{Special Joyce structures and the associated hyperk\"ahler structure}

We take as starting point an affine special K\"{a}hler manifold $(M,I,\omega,\nabla)$ and let $N=TM$ with the canonical projection $\pi:N\to M$. In the following, we introduce several structures required to define special Joyce structures.

\subsubsection{The induced symplectic structure on the vertical bundle}\label{vertsympstructure}

We first describe the symplectic structure on the vector bundle $V_{\pi}\to N$ induced from the ASK structure on $M$. The same discussion holds when we consider the complexified bundle $V_{\pi}\otimes \mathbb{C}\to N$\,.\\

On $\pi:TM \to M$ we have a flat Ehresmann connection 
\begin{equation}
    \mathcal{H}: \pi^*(TM)\to TN
\end{equation}
induced from the flat and torsion-free connection $\nabla$ via the corresponding parallel transport (recall the end of Section \ref{complexEhressec}). In terms of affine special coordinates $(x^i,y_i)$ on $M$ (which in particular are flat with respect to $\nabla$) and the induced coordinates $(x^i,y_i,\varphi^i,\varphi_i)$ on $N=TM$, it follows from \eqref{inducedconn} that
\begin{equation}\label{flatlift}
    \mathcal{H}_{\frac{\partial}{\partial x^i}}=\frac{\partial}{\partial x^i}, \quad \mathcal{H}_{\frac{\partial}{\partial y_i}}=\frac{\partial}{\partial y_i}\,.
\end{equation}
If we consider a conjugate system of holomorphic special coordinates $(Z^i)$ and $(Z_i)$ inducing the affine special coordinates $(x^i,y_i)$ (recall Lemma \ref{coordlemma}), we conclude from \eqref{Zexp} and \eqref{flatlift} that
\begin{equation}\label{hholspecial}
    \mathcal{H}_{\frac{\partial}{\partial Z^i}}=\frac{\partial}{\partial Z^i}\,.
\end{equation}
Furthermore, we have the canonical identification of vector bundles over $N$
\begin{equation}
    \nu:\pi^*(TM)\to V_{\pi}=\text{Ker}(\mathrm{d}\pi)\,
\end{equation}
given by
\begin{equation}
    \nu(V_p,W_p)=\frac{\mathrm{d}}{\mathrm{d}t}(V_p+tW_p)\big|_{t=0}\in V_{\pi}|_{V_p}\,.
\end{equation}
We will use the same notation that we use for the evaluation of Ehresmann connections and for $X\in \pi^*(TM)$ denote $\nu_X:=\nu(X)$. As in the case of Ehresmann connections, we will also sometimes abuse notation and evaluate $\nu$ on local sections of $TM \to M$, with the understanding that we evaluate it on the canonical pullback section. In terms of the above coordinates $(x^i,y_i,\varphi^i,\varphi_i)$ on $N$, we have 
\begin{equation}
    \nu_{\frac{\partial}{\partial x^i}}=\frac{\partial}{\partial \varphi^i}, \quad \nu_{\frac{\partial}{\partial y_i}}=\frac{\partial}{\partial \varphi_i}\,,
\end{equation}
while for future reference, we note that from \eqref{Zexp} we find that the complex linear extension of $\nu$ satisfies
\begin{equation}\label{nuholspecial}
    \nu_{\frac{\partial}{\partial Z^i}}=\frac{1}{2}\left(\frac{\partial}{\partial \varphi^i}-\tau_{ij}\frac{\partial}{\partial \varphi_j}\right)\,.
\end{equation}
We now use $\nu$ to induce from $\omega$ a  symplectic structure $\omega^{\nu}$ on the vector bundle $V_{\pi}\to N$. More precisely, letting $p_2:\pi^*(TM)\to TM$ denote the projection into the second factor $p_2(X_p,W_p)=W_p$, and given $V_1,V_2\in V_{\pi}|_{X_p}$ where $X_p\in T_pM\subset N$ we define 
\begin{equation}\label{vertomega}
\omega^{\nu}_{X_p}(V_1,V_2):=\omega_p(p_2( \nu^{-1}(V_1)),p_2(\nu^{-1}(V_2)))\,.    
\end{equation}
In terms of affine special coordinates $(x^i,y_i)$ for $(M,I,\omega,\nabla)$ where

\begin{equation}\label{omegadarboux}
    \omega=\mathrm{d}x^i\wedge \mathrm{d}y_i\,,
\end{equation}
one easily checks that with respect to the induced coordinates $(x^i,y_i,\varphi^i,\varphi_i)$ on $N$,  we can write
\begin{equation}\label{omegavertdarboux}
    \omega^{\nu}=\mathrm{d}\varphi^i\wedge \mathrm{d}\varphi_i\,.
\end{equation}

In what follows, whenever we consider Hamiltonian vector fields (denoted by $\text{Ham}_f$) or Poisson brackets (denoted by $\{-,-\}$), we do with respect to the symplectic structure $\omega^{\nu}$ on $V_{\pi}\to N$. Namely, for $V\in V_{\pi}\subset TN$ and $f$,$g$ functions on $N$ we let 
\begin{equation}\label{usefulexp2}
    Vf=\omega^{\nu}(V,\text{Ham}_f), \quad \{f,g\}=\omega^{\nu}(\text{Ham}_f,\text{Ham}_g)\,.
\end{equation}
More concretely, in terms of the coordinates $(x^i,y_i,\varphi^i,\varphi_i)$ on $N$ from before 
\begin{equation}\label{usefulexp}
    \text{Ham}_f=\frac{\partial f}{\partial \varphi^i}\frac{\partial}{\partial \varphi_i}-\frac{\partial f}{\partial \varphi_i}\frac{\partial}{\partial \varphi^i}\,, \quad \{f,g\}=\frac{\partial f}{\partial \varphi^i}\frac{\partial g}{\partial \varphi_i}-\frac{\partial f}{\partial \varphi_i}\frac{\partial g}{\partial \varphi^i}\,.
\end{equation}
We will also frequently use the relation
\begin{equation}\label{LiePoissonbracket}
    [\text{Ham}_f,\text{Ham}_g]=\text{Ham}_{\{f,g\}}\,.
\end{equation}

Finally, given $X\in \pi^*(TM)|_{V_p}$ and a function $f$ on $N$, we will frequently write
\begin{equation}
    \text{Ham}_{\mathcal{H}_Xf}, \quad \text{Ham}_{\nu_Xf}\in V_{\pi}|_{V_p}\,.
\end{equation}
Since $\mathcal{H}_Xf$ and $\nu_Xf$ are just numbers, we clarify what we mean by the above expressions. First note that $X$ canonically extends to a section $\hat{X}$ of $p_1:\pi^*(TM)|_{T_pM}\to T_pM$, where $p_1:\pi^*(TM)\to TM$ is the canonical projection in the first factor. Indeed, if $X=(V_p,W_p)$, then we have
\begin{equation}
    \hat{X}_{Z_p}=(Z_p,W_p), \quad Z_p\in T_pM\,.
\end{equation}
We then have that $\mathcal{H}_{\hat{X}}f$ and $\nu_{\hat{X}}f$ are functions on $T_pM$, and hence we can compute the Hamiltonian vector fields $\text{Ham}_{\mathcal{H}_{\hat{X}}f}$, and  $\text{Ham}_{\nu_{\hat{X}}f}$ on $T_pM$. We then set 
\begin{equation}\label{hamnot}
    \text{Ham}_{\mathcal{H}_Xf}:=\text{Ham}_{\mathcal{H}_{\hat{X}}f}|_{V_p}, \quad \text{Ham}_{\nu_Xf}:=\text{Ham}_{\nu_{\hat{X}}f}|_{V_p}\,.
\end{equation}
Again, in terms of the coordinates $(x^i,y_i,\varphi^i,\varphi_i)$ from above, if 
\begin{equation}
    X=\left(V_p\;,\;W^i\frac{\partial}{\partial x^i}\bigg|_{p}+W_i\frac{\partial}{\partial y_i}\bigg|_{p}\right) 
\end{equation}
then
\begin{equation}\label{verthamexp}
\begin{split}
    \text{Ham}_{\mathcal{H}_Xf}&=W^i\left(\frac{\partial^2f}{\partial \varphi^j\partial x^i}(V_p)\frac{\partial}{\partial \varphi_j}\bigg|_{V_p}-\frac{\partial^2f}{\partial \varphi_j\partial x^i}(V_p)\frac{\partial}{\partial \varphi^j}\bigg|_{V_p}\right)\\
    &\quad +W_i\left(\frac{\partial^2f}{\partial \varphi^j\partial y_i}(V_p)\frac{\partial}{\partial \varphi_j}\bigg|_{V_p}-\frac{\partial^2f}{\partial \varphi_j\partial y_i}(V_p)\frac{\partial}{\partial \varphi^j}\bigg|_{V_p}\right)\\
\text{Ham}_{\nu_Xf}&=W^i\left(\frac{\partial^2f}{\partial \varphi^j\partial \varphi^i}(V_p)\frac{\partial}{\partial \varphi_j}\bigg|_{V_p}-\frac{\partial^2f}{\partial \varphi_j\partial \varphi^i}(V_p)\frac{\partial}{\partial \varphi^j}\bigg|_{V_p}\right)\\
    &\quad +W_i\left(\frac{\partial^2f}{\partial \varphi^j\partial \varphi_i}(V_p)\frac{\partial}{\partial \varphi_j}\bigg|_{V_p}-\frac{\partial^2f}{\partial \varphi_j\partial \varphi_i}(V_p)\frac{\partial}{\partial \varphi^j}\bigg|_{V_p}\right)\,.\\
\end{split}
\end{equation}
Note that the above quantities only depend on the vector $X\in \pi^*(TM)|_{V_p}$ instead of $\hat{X}$, which justifies the notation in \eqref{hamnot}.
\subsubsection{Special Joyce structures}

We now specialize the family $\mathcal{A}^{\zeta}$ of complexified Ehresmann connections given in \eqref{Ehrcongen}. The particular form of the maps $h$ and $v$ defined in \eqref{hdef} and \eqref{vdef} is motivated by trying to generalize what happens in our main example in Section \ref{uncoupled}.
\begin{definition}
    Given an ASK manifold $(M,g,\omega,\nabla)$, an almost special Joyce structure over $M$ is the data of a family of complexified Ehresmann connections $\mathcal{A}^{\zeta}$ of the form \eqref{Ehrcongen} such that $h$ and $v$ satisfy \eqref{non-deg}, and such that for some $J:N\to \mathbb{C}$ we have\footnote{Recall that $v$ is complex anti-linear and the discussion in Section \ref{vertsympstructure} regarding the expressions $\text{Ham}_{\mathcal{H}_XJ}$ and $\text{Ham}_{\nu_{\overline{X}}J}$.} 
    \begin{equation}\label{hvdef}
    h_X=\mathcal{H}_X+\text{Ham}_{\mathcal{H}_XJ}, \quad v_X=2\pi \mathrm{i}(\nu_{\overline{X}}+\text{Ham}_{\nu_{\overline{X}}J}), \quad X\in \pi^*(T^{1,0}M)\,\,.
\end{equation}
Namely, the family $\mathcal{A}^{\zeta}$ is given by 
    \begin{equation}\label{famconn}
\begin{split}
\mathcal{A}_{X}^{\zeta}&=\mathcal{H}_X+\text{Ham}_{\mathcal{H}_XJ}+\frac{2\pi\mathrm{i} }{\zeta}\left(\nu_{X} + \text{Ham}_{\nu_{X}\overline{J}}\right)\\
\mathcal{A}_{\overline{X}}^{\zeta}&=\mathcal{H}_{\overline{X}}+\text{Ham}_{\mathcal{H}_{\overline{X}}\overline{J}}+2\pi\mathrm{i}\zeta\left(\nu_{\overline{X}} +\text{Ham}_{\nu_{\overline{X}}J}\right)\,.
\end{split}
\end{equation}

\end{definition}
The expressions \eqref{famconn} should be compared with the analogous but simpler expression in \cite[Equation 40]{BridgeJoyce}. There, if $(\mathcal{M},\Omega, \nabla)$ is a holomorphic symplectic manifold with a compatible flat, torsion-free connection $\nabla$, the relevant family of (holomorphic) Ehresmann connections $\widetilde{\mathcal{A}}^{\epsilon}$ on $\pi:T^{1,0}\mathcal{M}\to \mathcal{M}$, $\epsilon \in \mathbb{C}^{\times}$, has the form 
\begin{equation}\label{bridgelandpencil}
    \widetilde{\mathcal{A}}^{\epsilon}_X=\mathcal{H}_X +\text{Ham}_{\nu_X\widetilde{J}} + \epsilon^{-1}\nu_X\,, \quad X\in \pi^*(T^{1,0}\mathcal{M})\,,
\end{equation}
where now $\mathcal{H}$ is a flat holomorphic Ehresmann connection on $T\mathcal{M}\to \mathcal{M}$ induced from $\nabla$, $\widetilde{J}$ is a holomorphic function on $T\mathcal{M}$, and $\text{Ham}_{\nu_X\widetilde{J}}$ is computed again by an induced symplectic structure on the vertical bundle $\mathcal{V}_{\pi}\to T^{1,0}\mathcal{\mathcal{M}}$. Note in particular that $\widetilde{\mathcal{A}}^{\epsilon}$ is only defined for $X\in \pi^*(T^{1,0}\mathcal{M})$ and not for $\overline{X}\in \pi^*(T^{0,1}\mathcal{M})$\,.  

\begin{remark}\label{ambrem}
   Note that there are several functions $J:N\to \mathbb{C}$ specifying the same almost special Joyce structure $\mathcal{A}^{\zeta}$. In fact, it is clear for \eqref{famconn} that if $J$ specifies $\mathcal{A}^{\zeta}$, then all other functions specifying $\mathcal{A}^{\zeta}$ have the form $J+f$ where $f\in C^{\infty}(N,\mathbb{C})$ satisfies that
   \begin{equation}
    \text{Ham}_{\mathcal{H}_{X}f}=\text{Ham}_{\mathcal{H}_{\overline{X}}f}=\text{Ham}_{\nu_Xf}=\text{Ham}_{\nu_{\overline{X}}f}=0\,, \quad \text{for all} \quad X\in \pi^*(T^{1,0}M)\,.
   \end{equation}
   One easily checks using \eqref{verthamexp} that in terms of the coordinates $(x^i,y_i,\varphi^i,\varphi_i)$ on $N$ from above, such a function must have the local form
   \begin{equation}
f(x^i,y_i,\varphi^i,\varphi_i)=c_i\varphi^i+c^i\varphi_i + g(x^i,y_i), \quad \text{for some} \quad c_i,c^i\in \mathbb{C}\,.
   \end{equation}
   As in \cite{BridgeJoyce}, one could impose certain symmetries on $J$ to fix the above freedom. For example, in our examples in Section \ref{examplesec}, $J$ is invariant under the involution $\iota:N\to N$ given by $\iota(V_p)=-V_p$, which reduces the freedom of choosing $J$ to adding the pullback of a function from the base. In \cite{BridgeJoyce} a similar symmetry is imposed, where the corresponding $\widetilde{J}$ is odd under the involution $\iota$. 
\end{remark}

It will be convenient to write in detail the flatness condition for the family $\mathcal{A}^{\zeta}$ in the particular case of an almost special Joyce structure.

\begin{proposition}\label{flatspecial}
Consider an almost special Joyce structure $\mathcal{A}^{\zeta}$ over an ASK manifold $M$. Furthermore, let $J:N\to \mathbb{C}$ be any function describing $\mathcal{A}^{\zeta}$ as in \eqref{famconn}. Then the family of complexified Ehresmann connections $\mathcal{A}^{\zeta}$ is flat for all $\zeta\in \mathbb{C}^{\times}$ if and only if the following holds for all local holomorphic sections $X,Y$ of  $T^{1,0}M\to M$\footnote{Recall that in the expressions below we mean the evaluation of $\mathcal{H}$ and $\nu$ in the corresponding pullback local sections of $\pi^*(T^{1,0}M)\to N$}:
\begin{itemize}
\item The following local functions on $N=TM$ descend to the base $M$
    \begin{equation}\label{Plebanski-add}
    \begin{split}
        &\{\mathcal{H}_XJ,\mathcal{H}_YJ\},\quad    \{\nu_{\overline{X}}J,\nu_{\overline{Y}}J\}\,,\quad \{\mathcal{H}_XJ,\nu_{\overline{Y}}J\}\,.\\
    \end{split}
    \end{equation}
\item The following equations are satisfied, up to addition of functions that descend to the base $M$
\begin{equation}\label{Plebanski-like}
\begin{split}
    \nu_X(\mathcal{H}_{Y}(J-\overline{J}))-\nu_Y(\mathcal{H}_{X}(J-\overline{J}))&=\{\nu_Y\overline{J},\mathcal{H}_XJ\}-\{\nu_X\overline{J},\mathcal{H}_YJ\}\,\\
    \mathcal{H}_X(\mathcal{H}_{\overline{Y}}(J-\overline{J}))+4\pi^2 \cdot \nu_X(\nu_{\overline{Y}}(J-\overline{J}))&=\{\mathcal{H}_XJ,\mathcal{H}_{\overline{Y}}\overline{J}\} -4\pi^2 \{\nu_X\overline{J},\nu_{\overline{Y}}J\}\,.
\end{split}
\end{equation}
\end{itemize}
\end{proposition}
\begin{proof}
In what follows, we translate the flatness equations obtained in Lemma \ref{flatlemma} to the specific case of an almost special Joyce structure.\\

Using the flatness of $\mathcal{H}$ and \eqref{LiePoissonbracket}, it easily follows that 
\begin{equation}
[h_X,h_Y]=h_{[X,Y]}\iff [\text{Ham}_{\mathcal{H}_XJ},\text{Ham}_{\mathcal{H}_YJ}]=\text{Ham}_{\{\mathcal{H}_XJ,\mathcal{H}_YJ\}}=0\,.
\end{equation}
This in turn implies that $\{\mathcal{H}_XJ,\mathcal{H}_YJ\}$ must descend to a function on the base, since the Hamiltonian vector fields are taken with respect to the induced vertical symplectic structure.\\

On the other hand, using $[\nu_{\overline{X}},\nu_{\overline{Y}}]=0$, equation \eqref{LiePoissonbracket} and 
\begin{equation}
    [\nu_{\overline{X}},\text{Ham}_{\nu_{\overline{Y}}J}]-[\nu_{\overline{Y}},\text{Ham}_{\nu_{\overline{X}}J}], =\text{Ham}_{[\nu_{\overline{X}},\nu_{\overline{Y}}]J}=0
\end{equation}
it follows that
\begin{equation}
    [v_X,v_Y]=0 \iff \text{Ham}_{\{\nu_{\overline{X}}J,\nu_{\overline{Y}}J\}}=0
\end{equation}
so $\{\nu_{\overline{X}}J,\nu_{\overline{Y}}J\}$ must descend to a function on the base. \\

Now, if $(Z_i)$ and $(Z^i)$ are a conjugate system of holomorphic special coordinates on $M$, we have the relation
\begin{equation}
    \mathrm{d}Z_i=\tau_{ij}\mathrm{d}Z^j
\end{equation}
for holomorphic functions $\tau_{ij}$ on $M$ symmetric in $i$ and $j$ (recall Lemma \ref{usualsKident}). In particular, by taking exterior derivatives we obtain the identity
\begin{equation}
    \frac{\partial \tau_{ik}}{\partial Z^j}=\frac{\partial \tau_{jk}}{\partial Z^i}\,.
\end{equation}
Using the above, one can verify using \eqref{nuholspecial} and \eqref{hholspecial} that  
\begin{equation}\label{vcom}
    [\mathcal{H}_X,\nu_Y]-[\mathcal{H}_Y,\nu_X]=\nu_{[X,Y]}\,.
\end{equation}
Using the previous equation, we see that 
\begin{equation}
    [h_X,\overline{v_Y}]+[\overline{v_X},h_Y]=\overline{v_{[X,Y]}}
\end{equation}
reduces to 
\begin{equation}
    [\mathcal{H}_X,\text{Ham}_{\nu_Y\overline{J}}]-[\mathcal{H}_Y,\text{Ham}_{\nu_X{\overline{J}}}]+[\nu_{X} + \text{Ham}_{\nu_{X}\overline{J}},\text{Ham}_{\mathcal{H}_YJ}]-[\nu_Y+\text{Ham}_{\nu_{Y}\overline{J}},\text{Ham}_{\mathcal{H}_XJ}]=\text{Ham}_{\nu_{[X,Y]}{\overline{J}}}\;.
\end{equation}
Which can be rewritten using \eqref{vcom} and 
\begin{equation}
[\nu_X,\text{Ham}_{\mathcal{H}_Y\overline{J}}]=\text{Ham}_{\nu_X(\mathcal{H}_Y\overline{J})}, \quad [\mathcal{H}_X,\text{Ham}_{\nu_{Y}}\overline{J}]=\text{Ham}_{\mathcal{H}_X(\nu_Y\overline{J})}
\end{equation}
as 
\begin{equation}
    \text{Ham}_{\nu_X(\mathcal{H}_Y(J-\overline{J}))-\nu_Y(\mathcal{H}_X(J-\overline{J}))}=\text{Ham}_{\{\nu_Y\overline{J},\mathcal{H}_XJ\}-\{\nu_X\overline{J},\mathcal{H}_YJ\}}\,.\end{equation}
We then obtain that the following equality holds, up to additions of functions descending to the base $M$
\begin{equation}
    \nu_X(\mathcal{H}_{Y}(J-\overline{J}))-\nu_Y(\mathcal{H}_{X}(J-\overline{J}))=\{\nu_Y\overline{J},\mathcal{H}_XJ\}-\{\nu_X\overline{J},\mathcal{H}_YJ\}\,.
\end{equation}
Continuing with the equations from Lemma \ref{flatlemma}, using $[\mathcal{H}_X,\nu_{\overline{Y}}]=0$,  the equation
\begin{equation}
    [h_X,v_Y]=0\,,
\end{equation}
similarly gives that $\{\mathcal{H}_XJ,\nu_{\overline{Y}}J\}$ should descend to a function on the base.\\

Finally, using that $[\mathcal{H}_{X},\mathcal{H}_{\overline{Y}}]=0$ (since $X$ and $Y$ are holomorphic and $\mathcal{H}$ is flat) and $[\nu_X,\nu_{\overline{Y}}]=0$, the equation 
\begin{equation}
    [h_X,\overline{h_Y}]=[\overline{v_X},v_Y]
\end{equation}
reduces to the condition that, up to the addition of functions descending to the base $M$, 
\begin{equation}
\mathcal{H}_X(\mathcal{H}_{\overline{Y}}(J-\overline{J}))+4\pi^2 \cdot \nu_X(\nu_{\overline{Y}}(J-\overline{J}))=\{\mathcal{H}_XJ,\mathcal{H}_{\overline{Y}}\overline{J}\} -4\pi^2 \{\nu_X\overline{J},\nu_{\overline{Y}}J\}\,.
\end{equation}
\end{proof}

\begin{remark}\leavevmode
\begin{itemize}
    \item The equations \eqref{Plebanski-like} should be compared to the Pleba\'nski equations encoding the flatness of the family of holomorphic Ehresmann connections $\widetilde{\mathcal{A}}^{\epsilon}$ of \cite{BridgeJoyce}. Following the notation of \eqref{bridgelandpencil}, the flatness of $\widetilde{\mathcal{A}}^{\epsilon}$ reduces to imposing that the holomorphic function $\widetilde{J}$ should satisfy the following equation for local holomorphic sections $X,Y$ of $T^{1,0}\mathcal{M}\to \mathcal{M}$ up to addition of functions descending to the base $\mathcal{M}$ (see \cite[Equation 13]{BriStr})
    \begin{equation}\label{Plebanskieq}
        \nu_X(\mathcal{H}_Y\widetilde{J})-\nu_Y(\mathcal{H}_X\widetilde{J})=\{\nu_X\widetilde{J},\nu_Y\widetilde{J}\}\,.
    \end{equation}
In the above setting, $\widetilde{J}$ can always be redefined so that \eqref{Plebanskieq} is satisfied exactly (see below \cite[Equation 14]{BridgeJoyce}). To write \eqref{Plebanskieq} in coordinates, it is convenient to choose flat holomorphic Darboux coordinates $(Z^i,Z_i)$ on $\mathcal{M}$ and $(Z^i,Z_i,\varphi^i,\varphi_i)$ the induced holomorphic coordinates on $T^{1,0}\mathcal{M}$. Note that the coordinates $\varphi^i$, $\varphi_i$ are complex in this case. With respect to the above coordinates
\begin{equation}
\mathcal{H}_{\frac{\partial}{\partial Z^i}}=\frac{\partial}{\partial Z^i}, \quad \mathcal{H}_{\frac{\partial}{\partial Z_i}}=\frac{\partial}{\partial Z_i}, \quad \nu_{\frac{\partial}{\partial Z^i}}=\frac{\partial}{\partial \varphi^i}, \quad \nu_{\frac{\partial}{\partial Z_i}}=\frac{\partial}{\partial \varphi_i}\,,
\end{equation}
while the Poisson bracket is given by the induced vertical symplectic form $\mathrm{d}\varphi^i\wedge \mathrm{d}\varphi_i$. 
\item Similarly, if one wishes to write the equations in Proposition \ref{flatspecial} in coordinates, a convenient set of coordinates on $TM$ is given by $(x^i,y_i,\varphi^i,\varphi_i)$ as in the beginning of Section \ref{vertsympstructure}. Namely, $(x^i,y_i)$ are affine special coordinates on $M$ and $(x^i,y_i,\varphi^i,\varphi_i)$ the induced coordinates on $TM$. If $(Z^i)$ and $(Z_i)$ denote conjugate systems of holomorphic special coordinates on $M$ associated to $(x^i,y_i)$, then we have
\begin{equation}\label{coord}
    \mathcal{H}_{\frac{\partial}{\partial Z^i}}=\frac{\partial}{\partial Z^i}=\frac{1}{2}\left(\frac{\partial}{\partial x^i}-\tau_{ij}\frac{\partial}{\partial y_j}\right), \quad \nu_{\frac{\partial}{\partial Z^i}}=\frac{1}{2}\left(\frac{\partial}{\partial \varphi^i}-\tau_{ij}\frac{\partial}{\partial \varphi_j}\right), \quad \mathrm{d}Z_i=\tau_{ij}\mathrm{d}Z^j\,.
\end{equation}
Using \eqref{coord} together with \eqref{usefulexp} one can write explicitly the equations in 
Proposition \ref{flatspecial} in coordinates by taking $X=\partial_{Z^i}$, $Y=\partial_{Z^j}$ as holomorphic sections of $T^{1,0}M\to M$. Note that contrary to the case of \eqref{Plebanskieq}, there is a special geometry relation imposed in \eqref{coord}, and the coordinates $\varphi^i$, $\varphi_i$ are real. 
\end{itemize}
\end{remark}

Assume now that we have an almost special Joyce structure $\mathcal{A}^{\zeta}$ over $M$ such that the family $\mathcal{A}^{\zeta}$ is flat for all $\zeta\in \mathbb{C}^{\times}$. By Corollary \ref{hypercomplexprop} we obtain a hypercomplex structure $(N,I_1,I_2,I_3)$ on $N$. We now want to extend this hypercomplex structure to a particular hyperhermitian structure. That is, a tuple $(g,I_1,I_2,I_3)$ where $g$ is a (pseudo)-Riemannian metric, and $I_i$ are  complex structures satisfying the quaternionic relations and preserving $g$:
\begin{equation}
    g(I_i-,I_i-)=g(-,-)\,.
\end{equation}
For this, we will require the following lemma, which will also motivate the definition of special Joyce structure. 
\begin{lemma}\label{11obstruction}
    Consider the almost complex structure $I_3$ induced on $N$ by an almost special Joyce structure $\mathcal{A}^{\zeta}$. Then the almost complex structure $I_3$ on $N$ restricts to an almost complex structure on the bundle $V_{\pi}\to N$. Furthermore, the symplectic structure $\omega^{\nu}$ on $V_{\pi}\to N$ is of type $(1,1)$ with respect to $I_3$ if and only if for all local sections $X,Y$ of $T^{1,0}M\to M$ we have 
    \begin{equation}\label{compcond}
        \{\nu_{\overline{X}}J,\nu_{\overline{Y}}J\}=0\,.
    \end{equation}
\end{lemma}
\begin{proof}
The fact that $I_3$ restricts to an almost complex structure on $V_{\pi}\to N$ follows directly from \eqref{cstr}. Now consider the local holomorphic sections $\frac{\partial}{\partial Z^i}$ of $T^{1,0}M\to M$ given by holomorphic special coordinates on $M$. By \eqref{non-deg} we have that 
    \begin{equation}
        v_i:=v_{\frac{\partial}{\partial Z^i}}, \quad \overline{v_j}:=\overline{v_{\frac{\partial}{\partial Z^j}}}\,,
    \end{equation}
    give a local frame of $V_{\pi}\otimes \mathbb{C}\to N$. Furthermore, by \eqref{cstr} the $v_i$ (resp. $\overline{v_j}$) are of type $(1,0)$ (resp. $(0,1)$) with respect to $I_3$. Since $\omega^{\nu}$ is real, to check that it is of type $(1,1)$ with respect to $I_3$ it is enough to check that 
    \begin{equation}
        \omega^{\nu}(v_i,v_j)=0
    \end{equation}
    for all $i,j$. Denoting $\nu_{\overline{i}}=\nu_{\frac{\partial}{\partial \overline{Z}^i}}$, a direct computation using \eqref{hvdef}, \eqref{nuholspecial}, \eqref{omegavertdarboux} and \eqref{usefulexp2} shows that 
    \begin{equation}
    \begin{split}
        \omega^{\nu}(v_i,v_j)&=-4\pi^2\mathrm{d}\varphi^k\wedge\mathrm{d}\varphi_k(\nu_{\overline{i}}+\text{Ham}_{\nu_{\overline{i}}J},\nu_{\overline{j}}+\text{Ham}_{\nu_{\overline{j}}J})\\
        &=-\pi^2(\overline{\tau_{ij}}-\overline{\tau_{ji}})-4\pi^2[\nu_{\overline{i}},\nu_{\overline{j}}]J-4\pi^2\{\nu_{\overline{i}}J, \nu_{\overline{j}}J\}\\
        &=-4\pi^2\{\nu_{\overline{i}}J, \nu_{\overline{j}}J\}\,,\\
    \end{split}
    \end{equation}
    since $\tau_{ij}$ is symmetric and $[\nu_{\overline{i}},\nu_{\overline{j}}]=0$. The result then follows. 
\end{proof}
\begin{remark}
    Note that from Remark \ref{ambrem} it follows that \eqref{compcond} is independent of the function $J$ used to represent $\mathcal{A}^{\zeta}$, so it only depends on $\mathcal{A}^{\zeta}$.  Condition \eqref{compcond} can be thought as a compatibility condition between the ASK structure on $M$ and the almost special Joyce structure $\mathcal{A}^{\zeta}$. 
\end{remark}

The previous remark motivates the following\footnote{Ultimately, what really motivates this definition is that it guarantees the existence of an HK structure on TM, due to Theorem \ref{maintheorem}.}
\begin{definition}\label{specialJoycedef}
    Consider an almost special Joyce structure $\mathcal{A}^{\zeta}$ on $(M,I,\omega,\nabla)$. We say that $\mathcal{A}^{\zeta}$ is a special Joyce structure if 
    \begin{itemize}
        \item The family  $\mathcal{A}^{\zeta}$ is flat for each $\zeta\in \mathbb{C}^{\times}$.
        \item The symplectic structure $\omega^{\nu}$ on $V_{\pi}\to N$ is of type $(1,1)$ with respect to $I_3$.
    \end{itemize}
\end{definition}
\subsubsection{The hyperhermitian structure associated to a special Joyce structure}

Consider a special Joyce structure $\mathcal{A}^{\zeta}$ over $M$. We now extend the associated hypercomplex structure $(N,I_1,I_2,I_3)$ from Proposition \ref{almosthypercomplex} and Corollary \ref{hypercomplexprop} to a hyperhermitian structure $(N,g,I_1,I_2,I_3)$. To define $g$ we first define a real non-degenerate 2-form $\omega_3$ on $N$ as follows\footnote{We remark that the $-1/4\pi^2$ factor is only conventional, and it is chosen in order to match certain conventions when we discuss examples.}
\begin{equation}\label{omega3def}
    \omega_3(v_X,\overline{v_Y})=\omega_3(h_Y,\overline{h_X}):=-\frac{1}{4\pi^2}\omega^{\nu}(v_X,\overline{v_Y}), \quad X,Y\in \pi^*(T^{1,0}M)\,,
\end{equation}
with all the other pairings being $0$. From \eqref{omega3def}, \eqref{cstr} and the second condition in Definition \ref{specialJoycedef} it follows that $\omega_3$ is real, non-degenerate and preserved by $I_3$.

\begin{proposition}\label{hyperhermitianstr}
    Given the real two form $\omega_3$ as in \eqref{omega3def}, let 
    \begin{equation}
        g(-,-):=\omega_3(-,I_3-).
    \end{equation}
\end{proposition}
Then $g$ is a pseudo-Riemannian metric preserved by $I_i$, $i=1,2,3$. In particular, $(N,g,I_1,I_2,I_3)$ is an almost hyperhermitian manifold. 
\begin{proof}
Since $\omega_3$ is preserved by $I_3$, it follows that $g$ is a pseudo-Riemannian metric  preserved by $I_3$. To check that it is preserved by $I_1$, using \eqref{cstr} and \eqref{omega3def} note that 
\begin{equation}
    g(I_1h_X,I_1\overline{h_Y})=g(\overline{v_X},v_Y)=-\mathrm{i}\omega_3(v_Y,\overline{v_X})=-\mathrm{i}\omega_3(h_X,\overline{h_Y})=g(h_X,\overline{h_Y})\,,
\end{equation}
with an analogous computation showing that $g(I_1v_X,I_1\overline{v_Y})=g(v_X,\overline{v_Y})$. Similarly,
\begin{equation}
    g(I_1v_X,I_1\overline{h_Y})=g(-\overline{h_X},v_Y)=-\mathrm{i}\omega_3(\overline{h_X},v_Y)=0=-\mathrm{i}\omega_3(v_X,\overline{h_Y})=g(v_X,\overline{h_Y})\,.
\end{equation}
    The computations showing that $I_2$ preserves $g$ follow in the same way using \eqref{cstr} and \eqref{omega3def}. Finally, since the almost complex structures from \eqref{cstr} satisfy the quaterion relations, it follows that $(N,g,I_1,I_2,I_3)$ is almost hyperhermitian. 
\end{proof}

Besides $\omega_3$, we also define for $i=1,2$
\begin{equation}
    \omega_i(-,-):=g(I_i-,-)\,,
\end{equation}
and 
\begin{equation}
\Omega:=\omega_1+\mathrm{i}\omega_2\,.\,
\end{equation}
It is easy to check using \eqref{cstr} and 
\begin{equation}
    \Omega=\omega_3(-I_2+\mathrm{i}I_1-,-)
\end{equation}
that $\Omega$ is of type $(2,0)$ in the almost complex structure $I_3$, and that 
\begin{equation}\label{holomega_3}
    \Omega(h_X,h_Y)=\Omega(v_X,v_Y)=0, \quad \Omega(h_X,v_Y)=2\mathrm{i}\omega_3(\overline{v_X},v_Y)\,.
\end{equation}

\begin{definition}
    Given a special Joyce structure over an ASK manifold $M$, we call the hyperhermitian structure $(g,I_1,I_2,I_3)$ from Proposition \ref{hyperhermitianstr} the hyperhermitian structure associated to the special Joyce structure. 
\end{definition}

We would like to now show that the hyperhermitian structure $(N,g,I_1,I_2,I_3)$ is actually hyperk\"{a}hler. For this, it is enough to show that the forms $\omega_i$ are closed for $i=1,2,3$. Our main result is then as follows 
\begin{theorem}\label{maintheorem}
    Given a special Joyce structure $\mathcal{A}^{\zeta}$ over $(M,I,\omega,\nabla)$, the associated hyperhermitian structure $(N,g,I_1,I_2,I_3)$ is hyperk\"{a}hler. 
\end{theorem}
We give the proof of this theorem in the following section.
\subsubsection{Proof of the main theorem}

In order to prove Theorem \ref{maintheorem}, we discuss several preliminary results. 

\begin{lemma}\label{keyexp}
    For $X,Y$ local sections of $T^{1,0}M\to M$, we have the expression (recall that $v_X$ is anti-linear, while $\nu_X$ is linear)
\begin{equation}\label{omega3explicit}
    \omega_3(v_X,\overline{v_Y})=\omega(Y,\overline{X})+\nu_{\overline{X}}(\nu_Y(J-\overline{J})) - \left\{\nu_{\overline{X}}J,\nu_Y\overline{J}\right\}
\end{equation}
\end{lemma}
\begin{proof}
Consider $(Z^i)$ and $(Z_i)$ a conjugate system of holomorphic special coordinates, $(x^i,y_i)$ the associated affine special coordinates, and $(x^i,y_i,\varphi^i,\varphi_i)$ the induced coorinates on $N$. We compute locally with respect to holomorphic special coordinates $(Z^i)$ and the coordinates $(Z^i,\varphi^i,\varphi_i)$ on $N$. Denoting $v_i:=v_{\partial/\partial Z^i}$, $\nu_i:=\nu_{\partial/\partial Z^i}$, $\nu_{\overline{i}}=\nu_{\partial/\partial \overline{Z}^i}$, and using \eqref{omega3def}, \eqref{omegavertdarboux}, \eqref{hvdef}, \eqref{usefulexp2} and \eqref{nuholspecial} we find 
\begin{equation}\label{omega3verev}
    \begin{split}
        \omega_3(v_i,\overline{v_j})=-\frac{1}{4\pi^2}\omega^{\nu}(v_i,\overline{v_j})&=-\mathrm{d}\varphi^k\wedge \mathrm{d}\varphi_k(\nu_{\overline{i}}+\text{Ham}_{\nu_{\overline{i}}J},\nu_j+\text{Ham}_{\nu_j\overline{J}})\\
        &=\frac{\mathrm{i}}{2}\text{Im}(\tau_{ij})-\nu_{\overline{i}}\nu_j\overline{J}+\nu_j\nu_{\overline{i}}J-\{\nu_{\overline{i}}J,\nu_j\overline{J}\}\,\\
        &=\frac{\mathrm{i}}{2}\text{Im}(\tau_{ij})+\nu_{\overline{i}}\nu_j(J-\overline{J})-\{\nu_{\overline{i}}J,\nu_j\overline{J}\}\,,\\
    \end{split}
\end{equation}
where on the last line we used that $[\nu_j,\nu_{\overline{i}}]=0$.\\

The final expression \eqref{omega3explicit} then follows by taking into account that $v$ is complex anti-linear, while $\nu$ is complex linear.
\end{proof}
\begin{lemma}\label{closedsimp} The forms $\omega_i$ are closed if and only if with respect to any system of holomorphic special coordinates $(Z_i)$ on $M$ and the associated $h_i:=h_{\partial/\partial Z^i}$, $v_i:=v_{\partial/\partial Z^i}$
we have

\begin{equation}
    \begin{split}
v_i\omega_3(v_j,\overline{v_k})&=v_j\omega_3(v_i,\overline{v_k})\\
h_j\omega_3(v_i,\overline{v_k})&=\omega_3(v_i,[h_j,\overline{v_k}])\\
\omega_3([v_i,\overline{h_k}],v_j)&=\omega_3([v_j,\overline{h_k}],v_i)\\
v_i\omega_3(v_j,\overline{v_k})&=\omega_3(v_i,[v_j,\overline{v_k}])\\
    \end{split}
\end{equation}
\end{lemma}
\begin{proof}
The proof amounts to computing $\mathrm{d}\omega_3$ and $\mathrm{d}\Omega$, and then simplifying the equations obtained by setting $\mathrm{d}\omega_3=\mathrm{d}\Omega=0$ using the flatness conditions from Lemma \ref{flatlemma}, together with the definition of $\omega_3$ 
\eqref{omega3def} and the relation between $\Omega$ and $\omega_3$ in \eqref{holomega_3}. Since the computations are rather long (but simple), we write the detailed computations in Appendix \ref{lemmaappendix}.
\end{proof}

We now check the above equations one by one. We recall that in the computations below the Hamiltonian vector fields and Poisson bracket are with respect to the vertical symplectic structure from Section \ref{vertsympstructure}.

\begin{lemma}\label{cond1}
    Given a special Joyce-structure, the equation 
    \begin{equation}
v_i\omega_3(v_j,\overline{v_k})=v_j\omega_3(v_i,\overline{v_k})
    \end{equation}
    holds.
\end{lemma}
\begin{proof}

By Lemma \ref{keyexp} and \eqref{hvdef} we obtain
\begin{equation}
\begin{split}
v_i\omega_3(v_j,\overline{v_k})&=2\pi \mathrm{i}(\nu_{\overline{i}}+\text{Ham}_{\nu_{\overline{i}}J})(\nu_{\overline{j}}\nu_k(J-\overline{J}) -\{\nu_{\overline{j}}J,\nu_k\overline{J}\})\,\\
&=2\pi \mathrm{i}(\nu_{\overline{i}}\nu_{\overline{j}}\nu_k(J-\overline{J})+(-\nu_{\overline{i}}\{\nu_{\overline{j}}J,\nu_k\overline{J}\}+\text{Ham}_{\nu_{\overline{i}}J}(\nu_{\overline{j}}\nu_k(J-\overline{J})))\\
&\quad -\text{Ham}_{\nu_{\overline{i}}J}\{\nu_{\overline{j}}J,\nu_k\overline{J}\})
\end{split}
\end{equation}
Now note that since $[\nu_{\overline{i}},\nu_{\overline{j}}]=0$ (recall \eqref{nuholspecial}), the quantity 
\begin{equation}
\nu_{\overline{i}}\nu_{\overline{j}}\nu_k(J-\overline{J})
\end{equation}
is symmetric in $i$ and $j$. On the other hand, using \eqref{usefulexp2} and the Jacobi identity for the Poisson bracket
\begin{equation}
\begin{split}
\text{Ham}_{\nu_{\overline{i}}J}\{\nu_{\overline{j}}J,\nu_k\overline{J}\}&=\{\nu_{\overline{i}}J,\{\nu_{\overline{j}}J,\nu_k\overline{J}\}\}\\
&=-\{\nu_{\overline{j}}J,\{\nu_{k}\overline{J},\nu_{\overline{i}}J\}\}-\{\nu_{k}\overline{J},\{\nu_{\overline{i}}J,\nu_{\overline{j}}J\}\}\\
&=\text{Ham}_{\nu_{\overline{j}}J}\{\nu_{\overline{i}}J,\nu_{k}\overline{J}\}-\{\nu_{k}\overline{J},\{\nu_{\overline{i}}J,\nu_{\overline{j}}J\}\}\,.\\
\end{split}
\end{equation}
Recalling from Proposition \ref{flatspecial} that  
\begin{equation}
    \{\nu_{\overline{i}}J,\nu_{\overline{j}}J\}
\end{equation}
must be a function that descends to the base (in fact is must be $0$ by the second property of a special Joyce structure), it follows that 
\begin{equation}
   \{\nu_{k}\overline{J},\{\nu_{\overline{i}}J,\nu_{\overline{j}}J\}\}=0 \implies \text{Ham}_{\nu_{\overline{i}}J}\{\nu_{\overline{j}}J,\nu_{k}\overline{J}\}=\text{Ham}_{\nu_{\overline{j}}J}\{\nu_{\overline{i}}J,\nu_{k}\overline{J}\}\,. 
\end{equation}
On the other hand, by using again that $[\nu_{\overline{i}},\nu_{\overline{j}}]=[\nu_i,\nu_{\overline{j}}]=0$ and $\{\nu_{\overline{i}}J,\nu_{\overline{j}}J\}=0$ it easily follows by an explicit computation that 
\begin{equation}
    -\nu_{\overline{i}}\{\nu_{\overline{j}}J,\nu_k\overline{J}\}+\text{Ham}_{\nu_{\overline{i}}J}(\nu_{\overline{j}}\nu_k(J-\overline{J}))=-\nu_{\overline{i}}\{\nu_{\overline{j}}J,\nu_k\overline{J}\}+\{\nu_{\overline{i}}J,\nu_{\overline{j}}\nu_k(J-\overline{J})\}
\end{equation}
is symmetric in $i$ and $j$. Hence we conclude that 
    \begin{equation}
v_i\omega_3(v_j,\overline{v_k})=v_j\omega_3(v_i,\overline{v_k})
    \end{equation}
\end{proof}
\begin{lemma}\label{cond2}
   Given a special Joyce structure, the following holds
\begin{equation}
v_i\omega_3(v_j,\overline{v_k})=\omega_3(v_i,[v_j,\overline{v_k}])\,.
\end{equation}
\end{lemma}
\begin{proof}
On one hand, by an explicit computation using \eqref{hvdef} and \eqref{LiePoissonbracket} we obtain
\begin{equation}
    [v_j,\overline{v_k}]=4\pi^2\text{Ham}_{-\nu_{\overline{j}}\nu_k(J-\overline{J})+\{\nu_{\overline{j}}J,\nu_k\overline{J}\}}=-4\pi^2\text{Ham}_{J_{\overline{j}k}}
\end{equation}
where we have set
\begin{equation}
    J_{\overline{j}k}:=\nu_{\overline{j}}\nu_k(J-\overline{J})+\{\nu_k\overline{J},\nu_{\overline{j}}J\}\,.
\end{equation}
On the other hand, we have 
\begin{equation}
\begin{split}
    \omega_3(v_i,[v_j,\overline{v_k}])&=-(2\pi)^3\mathrm{i}(\omega_3(\nu_{\overline{i}},\text{Ham}_{J_{\overline{j}k}})+\omega_3(\text{Ham}_{\nu_{\overline{i}}J}, \text{Ham}_{J_{\overline{j}k}}))\,.
\end{split}
\end{equation}
Now note that by \eqref{omega3def}, \eqref{usefulexp} and \eqref{usefulexp2}
\begin{equation}
\begin{split}
-(2\pi)^3\mathrm{i}\omega_3(\text{Ham}_{\nu_{\overline{i}}J}, \text{Ham}_{J_{\overline{j}k}})&=2\pi \mathrm{i}\mathrm{d}\varphi^i\wedge \mathrm{d}\varphi_i(\text{Ham}_{\nu_{\overline{i}}J}, \text{Ham}_{J_{\overline{j}k}})\\
&=2\pi \mathrm{i}\{\nu_{\overline{i}}J, J_{\overline{j}k}\}\\
&=2\pi \mathrm{i}(\text{Ham}_{\nu_{\overline{i}}J}(\nu_{\overline{j}}\nu_k(J-\overline{J})) + \text{Ham}_{\nu_{\overline{i}}J}\{\nu_k\overline{J},\nu_{\overline{j}}J\})\,.
\end{split}
\end{equation}
Finally, an easy computation using again \eqref{omega3def}, \eqref{usefulexp} and \eqref{usefulexp2} shows that
\begin{equation}
    -(2\pi)^3\mathrm{i}\omega_3(\nu_{\overline{i}},\text{Ham}_{J_{\overline{j}k}})=2\pi \mathrm{i}\nu_{\overline{i}}J_{\overline{j}k}=2\pi \mathrm{i}(\nu_{\overline{i}}\nu_{\overline{j}}\nu_k(J-\overline{J})-\nu_{\overline{i}}\{\nu_{\overline{j}}J,\nu_k\overline{J}\})
\end{equation}
Hence, we conclude using Lemma \ref{keyexp} that 
\begin{equation}
v_i\omega_3(v_j,\overline{v_k})=\omega_3(v_i,[v_j,\overline{v_k}])\,.
\end{equation}
\end{proof}
\begin{lemma}\label{cond3}
    Give a special Joyce structure, the following holds
    \begin{equation}
h_j\omega_3(v_i,\overline{v_k})=\omega_3(v_i,[h_j,\overline{v_k}])
    \end{equation}
\end{lemma}
\begin{proof}

We start by noting that 
\begin{equation}
    [h_j,\overline{v_k}]=2\pi \mathrm{i}\left(-[\mathcal{H}_j,\nu_k]+\text{Ham}_{\mathcal{H}_j\nu_k(J-\overline{J})-\{\mathcal{H}_jJ,\nu_k\overline{J}\}-[\mathcal{H}_j,\nu_k]J}\right)
\end{equation}
Setting
\begin{equation}
J_{jk}:=\mathcal{H}_j\nu_k(J-\overline{J})-\{\mathcal{H}_jJ,\nu_k\overline{J}\}
\end{equation}
one finds that
\begin{equation}
\begin{split}
    \omega_3(v_i,[h_j,\overline{v_k}])&=4\pi^2\omega_3(\nu_{\overline{i}}+\text{Ham}_{\nu_{\overline{i}}J},[\mathcal{H}_j,\nu_k]-\text{Ham}_{J_{jk}-[\mathcal{H}_j,\nu_k]J})\,.
\end{split}
\end{equation}
We compute this by parts. On one hand, using as in the previous proposition, \eqref{omega3def}, \eqref{usefulexp} and \eqref{usefulexp2}, together with $[\mathcal{H}_j,\nu_{\overline{i}}]=0$ (since $\tau_{ij}$ is holomorphic), we find 
\begin{equation}
\begin{split}
    4\pi^2\omega_3(\nu_{\overline{i}},[\mathcal{H}_j,\nu_k]-\text{Ham}_{J_{jk}-[\mathcal{H}_j,\nu_k]J})&=4\pi^2\omega_3(\nu_{\overline{i}},[\mathcal{H}_j,\nu_k])-4\pi^2\omega_3(\nu_{\overline{i}},\text{Ham}_{J_{jk}-[\mathcal{H}_j,\nu_k]J})\\
    &=\frac{1}{4}\frac{\partial \tau_{ki}}{\partial Z^j}+(\nu_{\overline{i}}J_{jk}-\nu_{\overline{i}}[\mathcal{H}_j,\nu_k]J)\\
    &=\frac{1}{4}\frac{\partial \tau_{ki}}{\partial Z^j}+\mathcal{H}_j\nu_{\overline{i}}\nu_{k}(J-\overline{J})-\nu_{\overline{i}}\{\mathcal{H}_jJ,\nu_k\overline{J}\}-\nu_{\overline{i}}[\mathcal{H}_j,\nu_k]J\\
\end{split}
\end{equation}

On the other hand, using the Jacobi identity, together with the fact that $\{\nu_{\overline{i}}J,\mathcal{H}_jJ\}$ is a function that descends to the base due to Proposition \ref{flatspecial}, we find
\begin{equation}
    \begin{split}
-4\pi^2\omega_3(\text{Ham}_{\nu_{\overline{i}}J},\text{Ham}_{J_{jk}-[\mathcal{H}_j,\nu_k]J})&=\{\nu_{\overline{i}}J,J_{jk}-[\mathcal{H}_j,\nu_k]J\}\\
&=\{\nu_{\overline{i}}J,\mathcal{H}_j\nu_k(J-\overline{J})\}-\{\nu_{\overline{i}}J,\{\mathcal{H}_jJ,\nu_k\overline{J}\}\}-\{\nu_{\overline{i}}J,[\mathcal{H}_j,\nu_k]J\}\\
&=\{\nu_{\overline{i}}J,\mathcal{H}_j\nu_k(J-\overline{J})\}-\{\mathcal{H}_jJ,\{\nu_{\overline{i}}J,\nu_k\overline{J}\}\}-\{\nu_{\overline{i}}J,[\mathcal{H}_j,\nu_k]J\}\,.\\
\end{split}
\end{equation}
The remaining term gives, using that $[\nu_{\overline{i}},[\mathcal{H}_j,\nu_k]]=0$,
\begin{equation}
    4\pi^2\omega_3(\text{Ham}_{\nu_{\overline{i}}J},[\mathcal{H}_j,\nu_k])=[\mathcal{H}_j,\nu_k]\nu_{\overline{i}}J=\nu_{\overline{i}}[\mathcal{H}_j,\nu_k]J
\end{equation}

Combining everything we obtain 
\begin{equation}
    \begin{split}
        \omega_3(v_i,[h_j,\overline{v_k}])&=\frac{1}{4}\frac{\partial \tau_{ki}}{\partial Z^j}+\mathcal{H}_j\nu_{\overline{i}}\nu_k(J-\overline{J})-\{\mathcal{H}_jJ,\{\nu_{\overline{i}}J,\nu_k\overline{J}\}\}\\
        &\quad +\{\nu_{\overline{i}}J,\mathcal{H}_j\nu_k(J-\overline{J})\}-\nu_{\overline{i}}\{\mathcal{H}_jJ,\nu_k\overline{J}\}-\{\nu_{\overline{i}}J,[\mathcal{H}_j,\nu_k]J\}\,.
    \end{split}
\end{equation}
We rewrite the last three terms as follows, using that $[\mathcal{H}_j,\nu_{\overline{i}}]=[\nu_j,\nu_{\overline{i}}]=0$, and that $\{\nu_{\overline{i}}J,\mathcal{H}_jJ\}$ is a function that descends to the base
\begin{equation}
    \begin{split}
        &\{\nu_{\overline{i}}J,\mathcal{H}_j\nu_k(J-\overline{J})\}-\nu_{\overline{i}}\{\mathcal{H}_jJ,\nu_k\overline{J}\}-\{\nu_{\overline{i}}J,[\mathcal{H}_j,\nu_k]J\}\\
        &=\{\nu_{\overline{i}}J,\mathcal{H}_j\nu_kJ\}-\{\nu_{\overline{i}}J,\mathcal{H}_j\nu_k\overline{J}\}-\{\mathcal{H}_j\nu_{\overline{i}}J,\nu_k\overline{J}\}-\{\mathcal{H}_jJ,\nu_{\overline{i}}\nu_k\overline{J}\}-\{\nu_{\overline{i}}J,[\mathcal{H}_j,\nu_k]J\}\\
        &=\{\nu_{\overline{i}}J,\nu_k\mathcal{H}_jJ\}-\{\nu_{\overline{i}}J,\mathcal{H}_j\nu_k\overline{J}\}-\{\mathcal{H}_j\nu_{\overline{i}}J,\nu_k\overline{J}\}-\{\mathcal{H}_jJ,\nu_{\overline{i}}\nu_k\overline{J}\}\\
        &=-\{\nu_{\overline{i}}\nu_kJ,\mathcal{H}_jJ\}-\mathcal{H}_j\{\nu_{\overline{i}}J,\nu_k\overline{J}\}-\{\mathcal{H}_jJ,\nu_{\overline{i}}\nu_k\overline{J}\}\\
        &=-\mathcal{H}_j\{\nu_{\overline{i}}J,\nu_k\overline{J}\}+\{\mathcal{H}_jJ,\nu_{\overline{i}}\nu_k(J-\overline{J})\}\\
    \end{split}
\end{equation}
Hence, overall we have 
\begin{equation}
    \begin{split}
        \omega_3(v_i,[h_j,\overline{v_k}])&=\frac{1}{4}\frac{\partial \tau_{ki}}{\partial Z^j}+\mathcal{H}_j\nu_{\overline{i}}\nu_k(J-\overline{J})-\{\mathcal{H}_jJ,\{\nu_{\overline{i}}J,\nu_k\overline{J}\}\}\\
        &\quad -\mathcal{H}_j\{\nu_{\overline{i}}J,\nu_k\overline{J}\}+\{\mathcal{H}_jJ,\nu_{\overline{i}}\nu_k(J-\overline{J})\}\,.
    \end{split}
\end{equation}
The latter is easily seen to be equal to $h_j\omega_3(v_i,\overline{v_k})$, so
\begin{equation}
    h_j\omega_3(v_i,\overline{v_k})=\omega_3(v_i,[h_j,\overline{v_k}])\,.
\end{equation}
\end{proof}
\begin{lemma}\label{cond4}
    Given a special Joyce structure, the following holds
    \begin{equation}
\omega_3([v_i,\overline{h_k}],v_j)=\omega_3([v_j,\overline{h_k}],v_i)
    \end{equation}
\end{lemma}
\begin{proof}
    Following a similar computation to before, we now have
    \begin{equation}
        [v_i,\overline{h}_k]=2\pi \mathrm{i}([\nu_{\overline{i}},\mathcal{H}_{\overline{k}}]+\text{Ham}_{J_{\overline{i}\overline{k}}})
    \end{equation}
    where
    \begin{equation}
        J_{\overline{i}\overline{k}}:=\nu_{\overline{i}}\mathcal{H}_{\overline{k}}\overline{J}-\mathcal{H}_{\overline{k}}\nu_{\overline{i}}J + \{\nu_{\overline{i}}J,\mathcal{H}_{\overline{k}}\overline{J}\}\,.
    \end{equation}
    With this we find by similar computations from previous propositions
    \begin{equation}
    \begin{split}\omega_3([v_i,\overline{h_k}],v_j)&=-\frac{1}{4}\frac{\partial \overline{\tau_{ij}}}{\partial \overline{Z}^k}-\nu_{\overline{j}}J_{\overline{i}\overline{k}}+[\nu_{\overline{i}},\mathcal{H}_{\overline{k}}]\nu_{\overline{j}}J +\{J_{\overline{ik}},\nu_{\overline{j}}J\}\,.\\
    \end{split}
    \end{equation}
    Now note that the first term is symmetric in $i$ and $j$. Furthermore, expanding the next two terms we see that
    \begin{equation}
        -\nu_{\overline{j}}J_{\overline{ik}}+[\nu_{\overline{i}},\mathcal{H}_{\overline{k}}]\nu_{\overline{j}}J=-\nu_{\overline{j}}\nu_{\overline{i}}\mathcal{H}_{\overline{k}}\overline{J}+(\nu_{\overline{j}}\mathcal{H}_{\overline{k}}\nu_{\overline{i}}J+\nu_{\overline{i}}\mathcal{H}_{\overline{k}}\nu_{\overline{j}}J)-\mathcal{H}_{\overline{k}}\nu_{\overline{i}}\nu_{\overline{j}}J-\nu_{\overline{j}}\{\nu_{\overline{i}}J,\mathcal{H}_{\overline{k}}\overline{J}\}\,.
    \end{equation}
    In particular, using that $[\nu_{\overline{i}},\nu_{\overline{j}}]=0$ we see that all the terms above are symmetric in $i$ and $j$, except possibly the last. Hence, to check the identity that we want, we just need to show that the following expression is symmetric in $i$ and $j$
    \begin{equation}
        -\nu_{\overline{j}}\{\nu_{\overline{i}}J,\mathcal{H}_{\overline{k}}\overline{J}\}+\{J_{\overline{ik}},\nu_{\overline{j}}J\}\,.
    \end{equation}
    We have
    \begin{equation}
        \begin{split}
            -\nu_{\overline{j}}\{\nu_{\overline{i}}J,\mathcal{H}_{\overline{k}}\overline{J}\}+\{J_{\overline{ik}},\nu_{\overline{j}}J\}&=-\{\nu_{\overline{j}}\nu_{\overline{i}}J,\mathcal{H}_{\overline{k}}\overline{J}\}+(-\{\nu_{\overline{i}}J,\nu_{\overline{j}}\mathcal{H}_{\overline{k}}\overline{J}\}+\{\nu_{\overline{i}}\mathcal{H}_{\overline{k}}\overline{J},\nu_{\overline{j}}J\})-\{\mathcal{H}_{\overline{k}}\nu_{\overline{i}}J,\nu_{\overline{j}}J\}\\
            &\quad + \{\{\nu_{\overline{i}}J,\mathcal{H}_{\overline{k}}\overline{J}\}, \nu_{\overline{j}}J\}\,.
        \end{split}
    \end{equation}
    The first term and the term in parenthesis are clearly symmetric in $i$ and $j$. 
    On the other hand, $\{\nu_{\overline{i}}J,\nu_{\overline{j}}J\}=0$ by the Definition \ref{specialJoycedef} and Lemma \ref{11obstruction}, so we find that 
    \begin{equation}
        -\{\mathcal{H}_{\overline{k}}\nu_{\overline{i}}J,\nu_{\overline{j}}J\}=-\mathcal{H}_{\overline{k}}\{\nu_{\overline{i}}J,\nu_{\overline{j}}J\}+\{\nu_{\overline{i}}J,\mathcal{H}_{\overline{k}}\nu_{\overline{j}}J\}= -\{\mathcal{H}_{\overline{k}}\nu_{\overline{j}}J,\nu_{\overline{i}}J\}\,,
    \end{equation}
    while from the Jacobi identity and the fact $\{\nu_{\overline{i}}J,\nu_{\overline{j}}J\}=0$ it follows that 
    \begin{equation}
        \{\{\nu_{\overline{i}}J,\mathcal{H}_{\overline{k}}\overline{J}\},\nu_{\overline{j}}J\}=-\{\{\mathcal{H}_{\overline{k}}\overline{J},\nu_{\overline{j}}J,\},\nu_{\overline{i}}J\}-\{\{\nu_{\overline{j}}J,\nu_{\overline{i}}J\},\mathcal{H}_{\overline{k}}\overline{J}\}= \{\{\nu_{\overline{j}}J,\mathcal{H}_{\overline{k}}\overline{J}\},\nu_{\overline{i}}J\}
    \end{equation}
    so the last summand is also symmetric. The lemma then follows.
\end{proof}

From the previous Lemmas \ref{closedsimp}, \ref{cond1}, \ref{cond2}, \ref{cond3} and \ref{cond4} we find that the real $2$-forms $\omega_i$ for $i=1,2,3$ are closed.  In particular, it follows that $(N,g,I_1,I_2,I_3)$ is hyperk\"{a}hler and Theorem \ref{maintheorem} is proved. 
\section{Examples of special Joyce structures}\label{examplesec}
In this section we discuss two examples. One is the case of a trivial special Joyce structure, which recovers the semi-flat HK metric \cite{SK,ACD} associated to an ASK manifold; while the second concern HK metrics associated to uncoupled variations of BPS structures over an ASK manifold \cite{GMN,CT}. 
\subsection{The trivial special Joyce structure and the semi-flat HK metric}\label{trivialJoyce}

We start with the easiest case, where we take the family $\mathcal{A}^{\zeta}$ from \eqref{famconn} determined by 
\begin{equation}
    J=0\,.
\end{equation}
In this case, the bundle maps $h$ and $v$ from \eqref{hvdef} reduce to 
\begin{equation}
    h_X=\mathcal{H}_X, \quad v_X=2\pi \mathrm{i}\nu_{\overline{X}}
\end{equation}
and the family of complexified Ehresmann connections reduces to 
\begin{equation}\label{semiflatehresmann}
\begin{split}
\mathcal{A}_{X}^{\zeta}&=\mathcal{H}_X+\frac{2\pi\mathrm{i} }{\zeta}\nu_{X}\\
\mathcal{A}_{\overline{X}}^{\zeta}&=\mathcal{H}_{\overline{X}}+2\pi\mathrm{i}\zeta\nu_{\overline{X}}\,.
\end{split}
\end{equation}
It is easy to check that $h$ and $v$ from above satisfy \eqref{non-deg}. Furthermore, the flatness conditions of Proposition \ref{flatspecial} are trivially satisfied, while the second point of Definition \ref{specialJoycedef} follows from Lemma \ref{11obstruction}. Hence, the corresponding $\mathcal{A}^{\zeta}$ gives a special Joyce structure. \\

The K\"{a}hler form $\omega_3$ is simply given by (recall \eqref{omega3def} and Lemma \ref{keyexp})

\begin{equation}\label{omedared}
    4\pi^2\omega_3(\nu_{\overline{X}},\nu_{Y})=\omega_3(\mathcal{H}_Y,\mathcal{H}_{\overline{X}}):=\omega(Y,\overline{X}), \quad \omega_3(\mathcal{H}_X,\nu_{\overline{Y}})=\omega_3(\mathcal{H}_{\overline{X}},\nu_{Y})=0\,.
\end{equation}
where $X,Y$ are local sections of $T^{1,0}M\to M$. In particular, in terms of affine special coordinates $(x^i,y_i)$ on $M$ and the induced coordinates $(x^i,y_i,\varphi^i,\varphi_i)$ on $N=TM$, we obtain using \eqref{omega3def}, \eqref{omegadarboux}, \eqref{omegavertdarboux} and \eqref{omedared} that
\begin{equation}
    \omega_3=\mathrm{d}x^i\wedge \mathrm{d}y_i +\frac{1}{4\pi^2}\mathrm{d}\varphi_i\wedge \mathrm{d}\varphi^i\,.
\end{equation}
In terms of a conjugate system of holomorphic special coordinates $(Z^i)$ and $(Z_i)$ inducing $(x^i,y_i)$, we can further write using Lemma \ref{usualsKident}
\begin{equation}
    \omega_3=\frac{\mathrm{i}}{2}\text{Im}(\tau_{ij})\mathrm{d}Z^i\wedge \mathrm{d}\overline{Z}^j +\frac{1}{4\pi^2}\mathrm{d}\varphi_i\wedge \mathrm{d}\varphi^i\,.
\end{equation}
On the other hand, $\Omega=\omega_1+\mathrm{i}\omega_2$ can be determined by using \eqref{holomega_3}. One finds in local coordinates that
\begin{equation}
    \Omega=-\frac{1}{2\pi}\mathrm{d}Z^i\wedge (\mathrm{d}\varphi_i+\tau_{ij}\mathrm{d}\varphi^j)\,.
\end{equation}

Now in order to compare with the semi-flat HK structure, we remark that this structure is more naturally defined on $T^*M$ instead of $TM$. In order to relate them, we use the natural identification $TM\cong T^*M$ given by
\begin{equation}
    X\to \omega(X,-)\,,
\end{equation}
where $\omega$ is the K\"{a}hler form of the ASK manifold. \\

With such an identification, if $(x^i,y_i)$ are affine special coordinates on $M$, $(x^i,y_i,\varphi^i,\varphi_i)$ the induced coordinates on $TM$, and $(x^i,y_i,\theta_i,\theta^i)$ the coordinates induced on $T^*M$, then the identification by $\omega$ is given in terms of the above coordinates by
\begin{equation}
    (x^i,y_i,\varphi^i,\varphi_i)\to (x^i,y_i,-\varphi_i,\varphi^i)=(x^i,y_i,\theta_i,\theta^i)\,.
\end{equation}
In particular, in the above coordinates on $T^*M$ the induced HK structure on $T^*M$ via $\omega:TM\cong T^*M$ has the local form
\begin{equation}
    \omega_3=\mathrm{d}x^i\wedge \mathrm{d}y_i -\frac{1}{4\pi^2}\mathrm{d}\theta_i\wedge \mathrm{d}\theta^i\,, \quad \Omega=\frac{1}{2\pi}\mathrm{d}Z^i\wedge (\mathrm{d}\theta_i -\tau_{ij}\mathrm{d}\theta^j)\,.
\end{equation}
Comparing with the formulas from  \cite[Equation (2.25) and Equation (2.26)]{CT}, one checks that the induced HK structure on $T^*M$ matches the usual semi-flat HK structure associated to an ASK manifold. 

\subsection{HK metrics associated to uncoupled variations of BPS structures}\label{uncoupled}

Here we present our main non-trivial example. In order to present this example, we give a brief review of variation of BPS structures and a slight reformulation of certain ASK geometries.

\subsubsection{Special period structures and ASK geometries}\label{specialperiodsec}

Consider a complex manifold $M$ of dimension $\text{dim}_{\mathbb{C}}(M)=n$. 

\begin{definition}A period structure over a complex manifold $M$ is a tuple $(M,\Gamma,Z)$ such that 
\begin{itemize}
    \item $\Gamma\to M$ is a bundle of lattices. Furthermore $\Gamma$ has a fiberwise integral skew-pairing 
    \begin{equation}
            \langle -, - \rangle_p: \Gamma_p \times \Gamma_p \to \mathbb{Z}\,, \quad p\in M\,.
        \end{equation}
\item $Z$ is a holomorphic section of $\text{Hom}(\Gamma,\mathbb{C})\to M$. If $\gamma$ is a local section of $\Gamma$, we can contract $Z$ and $\gamma$ to obtain a local holomorphic function on $M$. We denote this contraction by
        \begin{equation}
            Z_{\gamma}:=Z(\gamma)\,.
        \end{equation}
\end{itemize}
\end{definition}

Now we explain the notion of special period structure over $M$. As we will see below, this should be thought as encoding an ASK structure on $M$ where the flat connection $\nabla$ comes from a certain ``integral" structure on $TM\to M$.
\begin{definition}\label{SPSdef}
    A period structure $(M,\Gamma,Z)$ is special if 
    \begin{itemize}
        \item $\Gamma\subset TM$ is a bundle lattices of rank $2n=\text{dim}(M)$ (i.e. $\Gamma_p\otimes \mathbb{R}=T_pM$ for all $p\in M$) and furthermore the pairing $\langle-,-\rangle$ is non-degenerate.
        We assume that around any point $p\in M$ we can find a local Darboux frame $(\gamma_i,\gamma^i)$ of $\langle -, - \rangle$. Namely
        \begin{equation}
            \langle \gamma_i,\gamma^j\rangle=\delta_i^j, \quad \langle \gamma_i,\gamma_j\rangle=\langle \gamma^i,\gamma^j\rangle=0\,.
        \end{equation}
        We denote by $\omega$ the non-degenerate 2-form on $M$ induced by $\langle -,-\rangle$.
        \item If $I$ is the complex structure of $M$, then $\omega$ is compatible with $I$ (i.e. $\omega(I-,I-)=\omega(-,-)$).
        \item If $\nabla$ is the flat connection on $M$ induced by $\Gamma \subset TM$ and $\xi^{1,0}$ is the complex vector field on $M$ determined by 
        \begin{equation}
            \frac{1}{2}Z=\omega(\xi^{1,0},-)\,,
        \end{equation}
        then we have
        \begin{equation}\label{holeuler}
            \pi^{1,0}=\nabla \xi^{1,0},
        \end{equation}
        where 
        \begin{equation}
            \pi^{1,0}:TM\otimes \mathbb{C}\to T^{1,0}M
        \end{equation}
        is the canonical projection into $(1,0)$ vectors with respect to $I$. 
    \end{itemize}
\end{definition}
\begin{proposition}\label{integralASK}
    Given a special period structure $(M,\Gamma,Z)$, one obtains an ASK structure $(M,\omega,\nabla)$ with K\"{a}hler form $\omega$ induced from the pairing $\langle-,-\rangle$ and flat connection $\nabla$ induced from $\Gamma$. Furthermore, given a local Darboux frame $(\gamma_i,\gamma^i)$ of $\Gamma\to M$, the associated holomorphic functions $(Z^i=Z_{\gamma^i})$, $(Z_i=Z_{\gamma_i})$ are conjugate systems of holomorphic special coordinates, and the affine special coordinate system $(x^i,y_i)$ induced from $(Z^i)$ and $(Z_i)$ satisfies that 
    \begin{equation}\label{flatcoords}
        \gamma_i=\partial_{x^i}, \quad \gamma^i=\partial_{y_i}\,.
    \end{equation}
\end{proposition}

\begin{proof}
    Since $\omega$ is induced from the pairing $\langle -,- \rangle$ on $\Gamma$, and $\nabla$ is induced from $\Gamma$, then we clearly have $\nabla \omega=0$\,.
    Now note that if $(\gamma_i,\gamma^i)$ is a local Darboux frame of $(\Gamma,\langle-,-\rangle)$ and $(\delta^i,\delta_i)$ denotes the dual frame on $\Gamma^*\subset T^*M$ then
    \begin{equation}
        Z=Z^i\delta_i +Z_i\delta^i, \quad \omega=\delta^i\wedge \delta_i
    \end{equation}
    implies that 
    \begin{equation}
        \xi^{1,0}=\frac{1}{2}(Z^i\gamma_i-Z_i\gamma^i)\,.
    \end{equation}
    Since $\pi^{1,0}=\nabla \xi^{1,0}$ and $2\text{Re}(\pi^{1,0})=\text{Id}_{TM}$, it follows that $(x^i=\text{Re}(Z^i), y_i=-\text{Re}(Z_i))$ is a coordinate system such that \eqref{flatcoords} holds, so in particular it must be a flat Darboux coordinate system for $\omega$, and hence and affine special coordinate system. We then obtain that 
    \begin{equation}
    \omega=\mathrm{d}x^i\wedge \mathrm{d}y_i\,,
    \end{equation}
so $\omega$ is closed, and hence symplectic. Since $\omega$ is compatible with $I$, we then obtain that $(M,I,\omega)$ is pseudo-K\"{a}her. Furthermore, the flatness of $\nabla$ implies $\mathrm{d}_{\nabla}^2=0$ (recall that $\mathrm{d}_{\nabla}$ denotes the extension of $\nabla$ to higher degree forms valued in $TM$, namely $\mathrm{d}_{\nabla}:\Omega^k(M,TM)\to \Omega^{k+1}(M,TM)$), and hence 
    \begin{equation}
        \pi^{1,0}=\frac{1}{2}(1_{TM}-\mathrm{i}I)=\nabla \xi^{1,0}=\mathrm{d}_{\nabla}\xi^{1,0}
    \end{equation}
    implies that
    \begin{equation}
        0=\mathrm{d}_{\nabla}\pi^{1,0}=\frac{1}{2}(\mathrm{d}_{\nabla}(1_{TM})-\mathrm{i}\mathrm{d}_{\nabla}(I)) \implies \mathrm{d}_{\nabla}(1_{TM})=0, \quad \mathrm{d}_{\nabla}(I)=0\,.
    \end{equation} 
    The condition $\mathrm{d}_{\nabla}(1_{TM})=0$ is equivalent to the torsion freeness of $\nabla$ (which already follows from previous arguments), while $\mathrm{d}_{\nabla}(I)=0$ is the remaining condition needed to obtain an ASK structure. Hence, we conclude that $(M,I,\omega,\nabla)$ is affine special K\"{a}hler. 
    It then follows from the same argument given in Lemma \ref{coordlemma} that $(Z^i)$ and $(Z_i)$ are conjugate systems of holomorphic special coordinates.
\end{proof}

\subsubsection{Variations of BPS structures}

Variations of BPS structures were introduced in \cite{VarBPS}. As mentioned in the introduction, they can be thought as abstracting certain natural structures associated to a triangulated 3d Calabi-Yau category, and its associated Donaldson-Thomas invariants and stability condition space. From the physics perspective, they can also be thought as abstracting natural structures associated to 4d $\mathcal{N}=2$ supersymmetric field theories and their BPS states \cite{GMN}.

\begin{definition}\label{varBPSdef} A variation of BPS structures is a tuple $(M,\Gamma,Z,\Omega)$ such that
\begin{itemize}
    \item $(M,\Gamma,Z)$ is a period structure.
    \item $\Omega: \Gamma\setminus \{0\} \to \mathbb{Q}$ is a function of sets satisfying the Kontsevich-Soibelman wall-crossing formula (see \cite{KS,VarBPS}) and $\Omega(\gamma)=\Omega(-\gamma)$. We will not state the wall-crossing formula, since it requires to introduce several notions and we will restrict to a simpler case where the statement of the wall-crossing formula becomes easier.\footnote{Roughly speaking, the numbers $\Omega(\gamma)$ jump along a real codimension $1$ subset of $M$ determined by $(M,\Gamma,Z)$. However, they do not jump arbitrarily, but the jump is uniquely determined by the Kontsevich-Soibelman wall-crossing formula. } The numbers $\Omega(\gamma)$ are called the BPS indices.  
    \item Support property: Given any compact set $K\subset M$ and a choice of covariant norm $|-|$ on $\Gamma|_K\otimes _{\mathbb{Z}}\mathbb{R}$, there should be a constant $C$ such that for any $\gamma \in \Gamma|_K\cap\text{Supp}(\Omega)$, we have
    \begin{equation}
        |Z_{\gamma}|>C|\gamma|\,.
    \end{equation}
    Here $\text{Supp}(\Omega)$ denotes the set of $\gamma \in \Gamma$ such that $\Omega(\gamma)\neq 0$.
    \item Convergence property: for any $R>0$, the series
    \begin{equation}
        \sum_{\gamma\in \Gamma|_p}|\Omega(\gamma)|e^{-R|Z_{\gamma}|}
    \end{equation}
    converges normally on compact subsets of $M$.
\end{itemize}
\end{definition}
\begin{remark}
    Our support and convergence property are stronger than as stated in \cite{VarBPS}, but the same as in \cite{CT}. We use this stronger version to guarantee that certain infinite sums involving the $\Omega(\gamma)$ below give rise to smooth functions. We do not rule out that weaker assumptions also guarantee the above. 
\end{remark}

\begin{definition}
    A variation of BPS structures $(M,\Gamma,Z,\Omega)$ is uncoupled (or mutually local) if for any $p\in M$ and  $\gamma,\gamma'\in \Gamma_p$
    \begin{equation}
        \Omega(\gamma),\Omega(\gamma')\neq 0 \implies \langle \gamma,\gamma'\rangle=0\,.
    \end{equation}
    We say that $(M,\Gamma,Z,\Omega)$ is coupled if it is not uncoupled.  
\end{definition}
In the case of an uncoupled variation of BPS structures, the wall-crossing formula implies that $\Omega(\gamma)$ must be locally constant and monodromy invariant.
\subsubsection{The special Joyce structure and the associated HK metric}\label{specialjoyceuncsec}

Our starting point is an uncoupled variation of BPS structures $(M,\Gamma,Z,\Omega)$ such that $(M,\Gamma,Z)$ is a special period structure. In particular, we have an ASK geometry on $M$ determined by $(M,\Gamma,Z)$ according to Proposition \ref{integralASK}. In \cite{CT} a hyperk\"{a}her structure on $T^*M$ is constructed from this data, based on previous work in the physics literature \cite{GMN}.\\

It is shown in \cite[Lemma 3.14]{CT} that in the above setting, we can always find a local Darboux frame $(\gamma_i,\gamma^i)$ of $\Gamma$ around any point $p\in M$ such that $\text{Supp}(\Omega)\subset \text{span}_{\mathbb{Z}}\{\gamma^i\}_{i=1,...,n}$. The frame $(\gamma_i,\gamma^i)$ induces a conjugate system of holomorphic special coordinates $(Z^i)$ and $(Z_i)$ on $M$ as in Proposition \ref{integralASK}, which in turn induces affine special coordinates $(x^i,y_i)$ on $M$. We denote as before $(x^i,y_i,\varphi^i,\varphi_i)$ the induced coordinates on $TM$ and $(x^i,y_i,\theta_i,\theta^i)$ the induced coordinates on $T^*M$. Furthermore, in such a situation if $\gamma\in \text{Supp}(\Omega)$ has the expression 
\begin{equation}
    \gamma=n_i(\gamma)\gamma^i,
\end{equation}
then we write
\begin{equation}
    \varphi_{\gamma}:=n_i(\gamma)\varphi^i\,,\quad  \theta_{\gamma}:=n_i(\gamma)\theta^i\,.
\end{equation}

Finally, in order to write the function $J$ specifying the special Joyce structure, we denote the modified Bessel functions of the second kind by $K_{\alpha}(x)$. The function $J$ is then defined by
\begin{equation}\label{Juncoupled}
    J=\frac{1}{2\pi \mathrm{i}}\sum_{\gamma}\Omega(\gamma)\sum_{n>0}\frac{e^{\mathrm{i}n\varphi_{\gamma}}}{n^2}K_0(2\pi n |Z_{\gamma}|)\,.
\end{equation}

In Appendix \ref{instgenapp} we discuss a relation between the function $J$ from \eqref{Juncoupled} and the instanton generating function $\mathcal{G}$ studied in \cite{BlackholeAP}, which in turn is related to the formula of the Pleba\'nski potential found in \cite{heavenlyJoyce}.
\begin{proposition}\label{Jim}
    The function $J$ given in \eqref{Juncoupled} defines a global smooth function on $N$ and it is imaginary valued. 
\end{proposition}
\begin{proof}
    By the same arguments given in \cite[Lemma 3.9]{CT}, the support property ensures that the summands are well defined (i.e. $|Z_{\gamma}|\neq 0$), while the normal convergence of the sum follows from the exponential decay of the Bessel functions $K_{\alpha}(x)$ as $x\to \infty$, together with the convergence property of the BPS structure and the support property (which guarantees that $|Z_{\gamma}|\to \infty$ as $||\gamma||\to \infty$). The normal convergence then implies that $J$ is smooth. Furthermore, the monodromy invariance of the $\Omega(\gamma)$'s and the fact that we sum over all $\gamma$ implies that the above expressions is actually a global function on $N$. Finally, the fact that $J$ is imaginary follows from the parity property $\Omega(\gamma)=\Omega(-\gamma)$.
\end{proof}
We can now use $J$ to define $\mathcal{A}^{\zeta}$ via \eqref{famconn}. The exponential decay of the terms involving $J$ as $|Z_{\gamma}|\to \infty$ (which follows from the same argument in \cite[Section 3.2]{CT}) and the fact that \eqref{semiflatehresmann} is a special Joyce structure implies that the non-degeneracy conditions \eqref{non-deg} are satisfied at least on $\pi^{-1}(U)\subset N$ for some open subset $U\subset M$. By restricting $M$ if necessary, we assume that $\eqref{non-deg}$ holds on all of $N=TM$.

\begin{proposition}
The family of complexified Ehresmann connections $\mathcal{A}^{\zeta}$ on $\pi:N \to M$ specified by the function $J$ via \eqref{famconn} defines a special Joyce structure.
\end{proposition}
\begin{proof}
    Consider as before a local Darboux frame $(\gamma_i,\gamma^i)$ of $\Gamma$  such that $\text{Supp}(\Omega)\subset \text{span}_{\mathbb{Z}}\{\gamma^i\}_{i=1,...,n}$. Furthermore, consider the corresponding induced conjugate system of holomorphic special coordinates $(Z^i)$ and $(Z_i)$ as in Proposition \ref{integralASK}, and the induced coordinate system $(x^i,y_i,\varphi^i,\varphi_i)$ on $N$. With respect to the coordinates $(Z^i,\varphi^i,\varphi_i)$ on $N$, we then find that the function $J$ only depends on $(Z^i,\varphi^i)$. This implies that all the expressions involving Poisson brackets in Proposition \ref{flatspecial} vanish, while the second condition of Definition \ref{specialJoycedef} is satisfied due to Lemma \ref{11obstruction}. Hence, using that $J=-\overline{J}$, the conditions from Proposition \ref{flatspecial} that remain to be checked is that for local holomorphic sections $X,Y$ of $T^{1,0}M\to M$
    \begin{equation}\label{Plebanski-simp}
\begin{split}
    \nu_X(\mathcal{H}_{Y}J)-\nu_Y(\mathcal{H}_{X}J)&=0\\
    \mathcal{H}_X(\mathcal{H}_{\overline{Y}}J)+4\pi^2 \cdot \nu_X(\nu_{\overline{Y}}J)&=0\,,
\end{split}
\end{equation}
up to the addition of functions that descend to $M$. In fact, we will see that these equations are satisfied exactly. To check this, we note that now we have
    \begin{equation}
        \mathcal{H}_{\frac{\partial}{\partial Z^i}}J=\frac{\partial J}{\partial Z^i}, \quad \nu_{\frac{\partial}{\partial Z^i}}J=\frac{1}{2}\frac{\partial J}{\partial \varphi^i}\,,
    \end{equation}
    where the latter is due to the fact that $J$ is independent of the $\varphi_i$. The equations \eqref{Plebanski-simp} then simplify to checking that 
    \begin{equation}\label{Plebanski-likelinear}
        \frac{\partial^2 J}{\partial Z^i \partial \varphi^j}=\frac{\partial^2 J}{\partial Z^j \partial \varphi^i}, \quad \frac{\partial^2 J}{\partial Z^i \partial \overline{Z}^j}+\pi^2\frac{\partial^2 J}{\partial \varphi^i \partial \varphi^j}=0\,.
    \end{equation}
     The first one follow easily, while the second follows from the following identities involving derivatives of the Bessel functions
    \begin{equation}
K'_{0}(x)=-K_1(x), \quad (xK_1(x))'=-xK_0(x)\,.
    \end{equation}
    We then conclude that the corresponding family $\mathcal{A}^{\zeta}$ defines a special Joyce structure.
\end{proof}

We now discuss the induced HK structure. The maps $h$ and $v$ from \eqref{hvdef} in this case reduce to
\begin{equation}
    h_{i}:=h_{\frac{\partial}{\partial Z^i}}=\frac{\partial}{\partial Z^i}+\frac{\partial^2J}{\partial Z^i \partial \varphi^j}\frac{\partial}{\partial \varphi_j}, \quad v_i:=v_{\frac{\partial}{\partial Z^i}}=\pi \mathrm{i}\left(\frac{\partial}{\partial \varphi^i}-\overline{\tau}_{ij}\frac{\partial}{\partial \varphi_j} + \frac{\partial^2 J}{\partial \varphi^i \partial \varphi^j}\frac{\partial}{\partial \varphi_j}\right)\,.
\end{equation}
On the other hand, from Lemma \ref{keyexp} and the fact that $J$ is imaginary-valued we have
\begin{equation}
    \omega_3(v_i,\overline{v_j})=\omega\left(\frac{\partial}{\partial Z^j},\frac{\partial}{\partial \overline{Z}^i}\right)+2\nu_{\frac{\partial}{\partial \overline{Z}^i}}\nu_{\frac{\partial}{\partial Z^j}}(J)=\frac{\mathrm{i}}{2}\text{Im}(\tau_{ij})+\frac{1}{2}\frac{\partial^2 J}{\partial \varphi^i \partial \varphi^j}\,.
\end{equation}
If we introduce the notation from \cite[Equation 3.9]{CT}
\begin{equation}
    \begin{split}
    V_{\gamma}&:=\frac{1}{2\pi}\sum_{n>0}e^{\mathrm{i}n\varphi_{\gamma}}K_0(2\pi n|Z_{\gamma}|)\\
    A_{\gamma}&:=-\frac{1}{4\pi}\sum_{n>0}e^{\mathrm{i}n\varphi_{\gamma}}|Z_{\gamma}|K_1(2\pi n|Z_{\gamma}|)\left(\frac{\mathrm{d}Z_{\gamma}}{Z_{\gamma}}-\frac{\mathrm{d}\overline{Z}_{\gamma}}{\overline{Z}_{\gamma}}\right)\,,
    \end{split}
\end{equation}
we then obtain
\begin{equation}\label{omega3uncopred}
\omega_3(v_i,\overline{v_j})=\omega_3(h_j,\overline{h_i})=\frac{\mathrm{i}}{2}\left(\text{Im}(\tau_{ij})+\sum_{\gamma}\Omega(\gamma)n_i(\gamma)n_j(\gamma)V_{\gamma}\right), \quad \omega_3(h_i,\overline{v_j})=\omega_3(v_j,\overline{h_i})=0\,.
\end{equation}
One can check by an explicit computation that we can write

\begin{equation}\label{omega3uncoupled}
    \omega_3=\frac{\mathrm{i}}{2}\text{Im}(\tau_{ij})\mathrm{d}Z^i\wedge \mathrm{d}\overline{Z}^j +\frac{1}{4\pi^2}\mathrm{d}\varphi_i\wedge \mathrm{d}\varphi^i+\sum_{\gamma}\Omega(\gamma)\left(\frac{\mathrm{i}}{2}V_{\gamma}\mathrm{d}Z_{\gamma}\wedge \mathrm{d}\overline{Z}_{\gamma}+\frac{1}{2\pi}\mathrm{d}\varphi_{\gamma}\wedge A_{\gamma}\right)\,\,,
\end{equation}
by evaluating the right hand side of \eqref{omega3uncoupled} on the local frame given by $h_i$, $\overline{h_i}$, $v_i$, $\overline{v_i}$ and checking that it matches with \eqref{omega3uncopred}. Similarly, using \eqref{holomega_3} one checks that the following expression for $\Omega$ holds
\begin{equation}
    \Omega=-\frac{1}{2\pi}\mathrm{d}Z^i\wedge (\mathrm{d}\varphi_i+\tau_{ij}\mathrm{d}\varphi^j)+\sum_{\gamma}\Omega(\gamma)\left(\mathrm{d}Z_{\gamma}\wedge A_{\gamma} +\frac{\mathrm{i}}{2\pi}V_{\gamma}\mathrm{d}\varphi_{\gamma}\wedge\mathrm{d}Z_{\gamma}\right)\,.
\end{equation}

As in the semi-flat case, using the identification $\omega:TM \cong T^*M$ given by $X\to \omega(X,-)$, we can induce from the above HK structure  an HK structure on $T^*M$. After doing so, we obtain in the coordinates $(Z^i,\theta_i,\theta^i)$ the follows local expressions for $\omega_3$ and $\Omega $
\begin{equation}
\begin{split}
    \Omega&=\frac{1}{2\pi}\mathrm{d}Z^i\wedge (\mathrm{d}\theta_i -\tau_{ij}\mathrm{d}\theta^j)+\sum_{\gamma}\Omega(\gamma)\left(\mathrm{d}Z_{\gamma}\wedge A_{\gamma} +\frac{\mathrm{i}}{2\pi}V_{\gamma}\mathrm{d}\theta_{\gamma}\wedge\mathrm{d}Z_{\gamma}\right)\\
    \omega_3&=\frac{\mathrm{i}}{2}\text{Im}(\tau_{ij})\mathrm{d}Z^i\wedge \mathrm{d}\overline{Z}^j -\frac{1}{4\pi^2}\mathrm{d}\theta_i\wedge \mathrm{d}\theta^i+\sum_{\gamma}\Omega(\gamma)\left(\frac{\mathrm{i}}{2}V_{\gamma}\mathrm{d}Z_{\gamma}\wedge \mathrm{d}\overline{Z}_{\gamma}+\frac{1}{2\pi}\mathrm{d}\theta_{\gamma}\wedge A_{\gamma}\right)\,,\\
\end{split}
\end{equation}
matching with the expressions from \cite[Equation (3.10) and (3.11)]{CT}. Hence this special Joyce structure recovers the HK metrics associated to uncoupled variations of BPS structures studied in \cite{CT}.

\section{Relation to hyperk\"{a}hler metrics on algebraic integrable systems}\label{intsysrel}

In this section we consider an ASK manifold $M$ whose ASK structure is given by a special period structure $(M,\Gamma,Z)$ (recall Definition \ref{SPSdef}), together with a special Joyce structure $\mathcal{A}^{\zeta}$ over $M$. We further assume a certain simple compatibility condition between $\mathcal{A}^{\zeta}$ and $(M,\Gamma,Z)$, introduced below. We then show that such special Joyce structures induce an HK metric with a compatible algebraic integrable system structure. 

\begin{definition}\label{percompdef}
Consider a special period structure $(M,\Gamma,Z)$ and a special Joyce structure $\mathcal{A}^{\zeta}$ over the induced ASK manifold $(M,I,\omega,\nabla)$. We say that $\mathcal{A}^{\zeta}$ is compatible with the period structure $(M,\Gamma,Z)$ if the induced hyperk\"{a}her structure on $N=TM$ is invariant under the action by fiberwise-translations by $2\pi\cdot \Gamma \subset TM$.\footnote{This type of condition is already included in the definition of Joyce structure in \cite{BridgeJoyce}.}
\end{definition}

In the above situation, we get an induced HK structure on the quotient $X:=TM/2\pi \cdot \Gamma$.  Note that if $\text{dim}_{\mathbb{C}}(M)=n$, then the fibers $X_p$ of the canonical projection $\pi:X\to M$ satisfy  
\begin{equation}
    X_p=T_pM/2\pi \cdot \Gamma_p\cong (S^1)^{2n}\,.
\end{equation}

If $(M,\Gamma,Z)$ is a special period structure, then special Joyce structures from our examples in Section \ref{examplesec} are compatible with $(M,\Gamma,Z)$. Indeed, note that if $(\gamma_i,\gamma^i)$ is a local Darboux frame of $\Gamma$, then by Proposition \ref{integralASK} we have that the induced affine special coordinates $(x^i,y_i)$ satisfy
\begin{equation}
    \gamma_i=\partial_{x^i}, \quad \gamma^i=\partial_{y_i}\,\,.
\end{equation}
Hence, with respect to the induced coordinates $(x^i,y_i,\varphi^i,\varphi_i)$ on $TM$, translations by $2\pi \cdot \Gamma$ amount to the shifts 
\begin{equation}
    \varphi^i\to \varphi^i+2\pi n^i, \quad \varphi_i\to \varphi_i+2\pi n_i, \quad n^i,n_i\in \mathbb{Z}\,.
\end{equation}
It is the easy to check from the expressions of the corresponding HK structures on Section \ref{examplesec} that the HK structures are invariant by translations by $2\pi \cdot \Gamma$.\\

We now recall the definition of an algebraic integrable system, taken from \cite[Section 3]{SK}.

\begin{definition}
An algebraic integrable system is a tuple $(\pi:X \to M, \Omega,[\rho])$ such that:
\begin{itemize}
\item $\pi: X \to M$ is a holomorphic submersion.
    \item  $\Omega\in \Omega^{2,0}(X)$ is a holomorphic symplectic form on $X$.
    \item The fibers $X_p:=\pi^{-1}(p)$ are compact and Lagrangian, and hence tori. 
    \item $[\rho]$ gives a family of smoothly varying classes $[\rho_p]\in H^{1,1}(X_p)\cap H^{2}(X_p,\mathbb{Z})$ defining a possibly indefinite polarization of $X_p$.
\end{itemize}
\end{definition}
\begin{definition}\label{compintHK}
    Consider a hyperk\"{a}hler manifold $(X,g,I_1,I_2,I_3)$ (possibly with indefinite signature) with associated K\"ahler forms $\omega_i$, $i=1,2,3$; together with an algebraic integrable system structure $(\pi:X\to M, \Omega,[\rho])$. We will say that the two structures are compatible if:
    \begin{itemize}
        \item $\pi:X\to M$ is holomorphic with respect to the complex structure $I_3$ on $X$.
        \item The holomorphic symplectic form $\omega_1+\mathrm{i}\omega_2$ with respect to $I_3$ equals $\Omega$.
        \item The polarizations of the fibers of $\pi:X\to M$ are specified by $\omega_3$ in the sense that 
        $[\omega_3|_{X_p}]=[\rho_p]$.
    \end{itemize}
\end{definition}

\begin{proposition}\label{HKintsysprop}
    Consider a special period structure $(M,\Gamma,Z)$ and a special Joyce structure $\mathcal{A}^{\zeta}$ over $M$ compatible with the period structure. Then the associated HK structure $(X:=TM/2\pi\cdot \Gamma,g,I_1,I_2,I_3)$ has a compatible algebraic integrable system structure where $\pi:X \to M$ is the canonical projection. 
\end{proposition}
\begin{proof}
    To show that $\pi$ is holomorphic with respect to $I_3$ and the complex structure $I$ on $M$, it is enough to show that 
    \begin{equation}
        \mathrm{d}\pi \circ I_3=I\circ \mathrm{d}\pi\,.
    \end{equation}
    The latter follows from \eqref{cstr}. On the other hand, the fact that the fibers are Lagrangian with respect to $\pi$ follows from \eqref{holomega_3}. Finally, note that from \eqref{omega3def}, we have that 
    \begin{equation}
        \omega_3|_{T_pM}=-\frac{1}{4\pi^2}\omega^{\nu}|_{T_pM}\,.
    \end{equation}
    In particular, with respect to affine special coordinates $(x^i,y_i)$ around $p$ and the induced coordinates $(x^i,y_i,\varphi^i,\varphi_i)$ on $TM$, we have (recall \eqref{omegavertdarboux})
    \begin{equation}
        \omega_3|_{T_pM}=\frac{1}{4\pi^2}\mathrm{d}\varphi_i\wedge \mathrm{d}\varphi^i\,.
    \end{equation}
    From this it follows that $\omega_3|_{X_p}$ is a closed form defining an integral cohomology class on $X_p\cong (S^{1})^{2n}$. The fact that it is of type $(1,1)$ follows from the second condition of Definition \ref{specialJoycedef}. Hence the cohomology class $[\omega_3|_{X_p}]$ defines a (possibly indefinite) polarization on $X_p$.  
\end{proof}

\begin{remark}\label{polrem}
    Note that the above HK geometries obtained from special Joyce structures satisfy a slightly stronger compatiblity condition than the one from Definition \ref{compintHK}. Namely, $\omega_3$ restricted to the fibers gives the unique invariant closed $(1,1)$ form specifying the polarizations, rather than just specifying the polarization via its cohomology class.    
\end{remark}

Note that by \cite[Theorem 3.4]{SK}, given an algebraic integrable system $(\pi:X\to M, \Omega,[\rho])$, there is an ASK structure on the base $M$ determined by the integrable system structure. Roughly speaking, $\Omega$ is used to relate $X$ to $T^*M/\Lambda$, where $\Lambda\subset T^*M$ is a bundle of full rank lattices. One then uses the bundle of lattices $\Lambda\to M$ to induce the flat connection on $M$ of the ASK structure, while the K\"{a}hler form $\omega$ is determined by the polarizations $[\rho_p]$. Suppose now that we start with an HK structure on $X=TM/(2\pi \cdot \Gamma)$ having a compatible algebraic integrable system structure $(\pi: X \to M, \Omega,[\rho])$. We further assume that the HK structure lifts to $TM$. By the previous argument, one automatically obtains an ASK structure on $M$. On the other hand, one can consider the $\mathbb{C}P^1$-family of complex structures $I_{\zeta}$ determined by the HK structure via \eqref{twistorhol} and the corresponding involutive distributions $T^{0,1}_{I_{\zeta}}(TM)\subset T(TM)\otimes \mathbb{C}$. The question is then whether the distributions $T^{0,1}_{I_{\zeta}}(TM)$ for $\zeta \in \mathbb{C}^{\times}$ allows us to define complexified Ehresmann connections $\mathcal{A}^{\zeta}$ with the form \eqref{famconn}. Whether this ``reverse" point of view on special Joyce structures holds will be deferred for future work.

\clearpage
  
\begin{appendix}
    \section{Proof of Lemma \ref{closedsimp}}\label{lemmaappendix}
        In general we have
\begin{equation}
    \mathrm{d}\omega_i(X,Y,Z)=X\omega_i(Y,Z)-Y\omega_i(X,Z)+Z\omega_i(X,Y) -\omega_i([X,Y],Z)+\omega_i([X,Z],Y)-\omega_i([Y,Z],X)\,.
\end{equation}
For $\omega_3$ in particular, since $\omega_3$ is of type $(1,1)$ in complex structure $I_3$, we know that $\mathrm{d}\omega_3$ can only have a $(2,1)$ and $(1,2)$ component. Using \eqref{omega3def} and Lemma \ref{flatlemma}, the relevant equations for the $(2,1)$ component are 
\begin{equation}
\begin{split}    \mathrm{d}\omega_3(h_i,h_j,\overline{h_k})&=h_i\omega_3(h_j,\overline{h_k})-h_j\omega_3(h_i,\overline{h_k})+\omega_3([h_i,\overline{h_k}],h_j)-\omega_3([h_j,\overline{h_k}],h_i)\\
\mathrm{d}\omega_3(v_i,h_j,\overline{h_k})&=v_i\omega_3(h_j,\overline{h_k})+\omega_3([v_i,\overline{h_k}],h_j)-\omega_3([h_j,\overline{h_k}],v_i)\\
\mathrm{d}\omega_3(v_i,v_j,\overline{h_k})&=\omega_3([v_i,\overline{h_k}],v_j)-\omega_3([v_j,\overline{h_k}],v_i)\\
\mathrm{d}\omega_3(h_i,h_j,\overline{v_k})&=\omega_3([h_i,\overline{v_k}],h_j)-\omega_3([h_j,\overline{v_k}],h_i)\\
\mathrm{d}\omega_3(v_i,h_j,\overline{v_k})&=-h_j\omega_3(v_i,\overline{v_k})+\omega_3([v_i,\overline{v_k}],h_j)-\omega_3([h_j,\overline{v_k}],v_i)\\
\mathrm{d}\omega_3(v_i,v_j,\overline{v_k})&=v_i\omega_3(v_j,\overline{v_k})-v_j\omega_3(v_i,\overline{v_k})+\omega_3([v_i,\overline{v_k}],v_j)-\omega_3([v_j,\overline{v_k}],v_i)
\end{split}
\end{equation}
for all $i,j,k=1,...,\text{dim}_{\mathbb{C}}(M)$.
The $(1,2)$ component is automatically obtained by conjugation and using the reality of $\omega_3$.\\

On the other hand, for $\Omega=\omega_1+\mathrm{i}\omega_2$, $\mathrm{d}\Omega$ has a $(3,0)$ and a $(2,1)$ component. Using \eqref{holomega_3} and Lemma \ref{flatlemma}, the equations relevant for the $(3,0)$ component are 
\begin{equation}
\begin{split}
\mathrm{d}\Omega(h_i,h_j,h_k)&=0\\
\mathrm{d}\Omega(v_i,h_j,h_k)&=-h_j\Omega(v_i,h_k)+h_k\Omega(v_i,h_j)\\
\mathrm{d}\Omega(v_i,v_j,h_k)&=v_i\Omega(v_j,h_k)-v_j\Omega(v_i,h_k)\\
\mathrm{d}\Omega(v_i,v_j,v_k)&=0
    \end{split}
\end{equation}
 while for the $(2,1)$ component we obtain
\begin{equation}
\begin{split}
\mathrm{d}\Omega(h_i,h_j,\overline{h_k})&=\Omega([h_i,\overline{h_k}],h_j)-\Omega([h_j,\overline{h_k}],h_i)\\
\mathrm{d}\Omega(v_i,h_j,\overline{h_k})&=\overline{h_k}\Omega(v_i,h_j)+\Omega([v_i,\overline{h_k}],h_j)-\Omega([h_j,\overline{h_k}],v_i)\\
\mathrm{d}\Omega(v_i,v_j,\overline{h_k})&=\Omega([v_i,\overline{h_k}],v_j)-\Omega([v_j,\overline{h_k}],v_i)\\
\mathrm{d}\Omega(h_i,h_j,\overline{v_k})&=\Omega([h_i,\overline{v_k}],h_j)-\Omega([h_j,\overline{v_k}],h_i)\\
\mathrm{d}\Omega(h_i,v_j,\overline{v_k})&=\overline{v_k}\Omega(h_i,v_j)+\Omega([h_i,\overline{v_k}],v_j)-\Omega([v_j,\overline{v_k}],h_i)\\
\mathrm{d}\Omega(v_i,v_j,\overline{v_k})&=\Omega([v_i,\overline{v_k}],v_j)-\Omega([v_j,\overline{v_k}],v_i)\,.\\
    \end{split}
\end{equation}
The above equations should hold again for all $i,j,k=1,...,\text{dim}_{\mathbb{C}}(M)$.\\

In order to simplify the above equations, we note that by an explicit computation using the definitions of $h$ and $v$ \eqref{hvdef}, one finds that  
\begin{equation}\label{vertbrackets}
    [h_i,\overline{h_j}], [v_i,\overline{h_j}], [h_i,\overline{v_j}],[v_i,\overline{v_j}]\in \text{span}\{v_i,\overline{v_j}\}_{i,j=1,...,\text{dim}_{\mathbb{C}}(M)}\,.
\end{equation}
 Using the above in conjunction with \eqref{omega3def}, the equations for $\mathrm{d}\omega_3$ simplify to
\begin{equation}
\begin{split}    \mathrm{d}\omega_3(h_i,h_j,\overline{h_k})&=h_i\omega_3(h_j,\overline{h_k})-h_j\omega_3(h_i,\overline{h_k})\\
\mathrm{d}\omega_3(v_i,h_j,\overline{h_k})&=v_i\omega_3(h_j,\overline{h_k})-\omega_3([h_j,\overline{h_k}],v_i)\\
\mathrm{d}\omega_3(v_i,v_j,\overline{h_k})&=\omega_3([v_i,\overline{h_k}],v_j)-\omega_3([v_j,\overline{h_k}],v_i)\\
\mathrm{d}\omega_3(h_i,h_j,\overline{v_k})&=0\\
\mathrm{d}\omega_3(v_i,h_j,\overline{v_k})&=-h_j\omega_3(v_i,\overline{v_k})-\omega_3([h_j,\overline{v_k}],v_i)\\
\mathrm{d}\omega_3(v_i,v_j,\overline{v_k})&=v_i\omega_3(v_j,\overline{v_k})-v_j\omega_3(v_i,\overline{v_k})+\omega_3([v_i,\overline{v_k}],v_j)-\omega_3([v_j,\overline{v_k}],v_i)
\end{split}
\end{equation}
while using \eqref{holomega_3} the ones for the $(2,1)$ component of $\mathrm{d}\Omega$ simplify to 

\begin{equation}
\begin{split}
\mathrm{d}\Omega(h_i,h_j,\overline{h_k})&=\Omega([h_i,\overline{h_k}],h_j)-\Omega([h_j,\overline{h_k}],h_i)\\
\mathrm{d}\Omega(v_i,h_j,\overline{h_k})&=\overline{h_k}\Omega(v_i,h_j)+\Omega([v_i,\overline{h_k}],h_j)\\
\mathrm{d}\Omega(v_i,v_j,\overline{h_k})&=0\\
\mathrm{d}\Omega(h_i,h_j,\overline{v_k})&=\Omega([h_i,\overline{v_k}],h_j)-\Omega([h_j,\overline{v_k}],h_i)\\
\mathrm{d}\Omega(h_i,v_j,\overline{v_k})&=\overline{v_k}\Omega(h_i,v_j)-\Omega([v_j,\overline{v_k}],h_i)\\
\mathrm{d}\Omega(v_i,v_j,\overline{v_k})&=0\\
    \end{split}
\end{equation}

Hence, overall we need to check if the following expressions are 0
\begin{equation}
\begin{split}    (1)\quad \mathrm{d}\omega_3(h_i,h_j,\overline{h_k})&=h_i\omega_3(h_j,\overline{h_k})-h_j\omega_3(h_i,\overline{h_k})\\
(2)\quad \mathrm{d}\omega_3(v_i,h_j,\overline{h_k})&=v_i\omega_3(h_j,\overline{h_k})-\omega_3([h_j,\overline{h_k}],v_i)\\
(3)\quad \mathrm{d}\omega_3(v_i,v_j,\overline{h_k})&=\omega_3([v_i,\overline{h_k}],v_j)-\omega_3([v_j,\overline{h_k}],v_i)\\
(4)\quad \mathrm{d}\omega_3(v_i,h_j,\overline{v_k})&=-h_j\omega_3(v_i,\overline{v_k})-\omega_3([h_j,\overline{v_k}],v_i)\\
(5)\quad \mathrm{d}\omega_3(v_i,v_j,\overline{v_k})&=v_i\omega_3(v_j,\overline{v_k})-v_j\omega_3(v_i,\overline{v_k})+\omega_3([v_i,\overline{v_k}],v_j)-\omega_3([v_j,\overline{v_k}],v_i)\\
(6)\quad\mathrm{d}\Omega(v_i,h_j,h_k)&=-h_j\Omega(v_i,h_k)+h_k\Omega(v_i,h_j)\\
(7)\quad \mathrm{d}\Omega(v_i,v_j,h_k)&=v_i\Omega(v_j,h_k)-v_j\Omega(v_i,h_k)\\
(8)\quad\mathrm{d}\Omega(h_i,h_j,\overline{h_k})&=\Omega([h_i,\overline{h_k}],h_j)-\Omega([h_j,\overline{h_k}],h_i)\\
(9)\quad\mathrm{d}\Omega(v_i,h_j,\overline{h_k})&=\overline{h_k}\Omega(v_i,h_j)+\Omega([v_i,\overline{h_k}],h_j)\\
(10)\quad\mathrm{d}\Omega(h_i,h_j,\overline{v_k})&=\Omega([h_i,\overline{v_k}],h_j)-\Omega([h_j,\overline{v_k}],h_i)\\
(11)\quad\mathrm{d}\Omega(h_i,v_j,\overline{v_k})&=\overline{v_k}\Omega(h_i,v_j)-\Omega([v_j,\overline{v_k}],h_i)\\
\end{split}
\end{equation}
As a next step to further simplify the above equations, we use \eqref{omega3def}, \eqref{holomega_3}, \eqref{vertbrackets}, and $[h_i,\overline{h_j}]=[\overline{v_i},v_j]$ from the flatness conditions in Lemma \ref{flatlemma}, obtaining (the equations with primes are the ones that were rewritten)
\begin{equation}
\begin{split} 
(1')\quad\mathrm{d}\omega_3(h_i,h_j,\overline{h_k})&=h_i\omega_3(v_k,\overline{v_j})-h_j\omega_3(v_k,\overline{v_i})\\
(2')\quad \mathrm{d}\omega_3(v_i,h_j,\overline{h_k})&=v_i\omega_3(v_k,\overline{v_j})-\omega_3([\overline{v_j},v_k],v_i)\\
(3)\quad \mathrm{d}\omega_3(v_i,v_j,\overline{h_k})&=\omega_3([v_i,\overline{h_k}],v_j)-\omega_3([v_j,\overline{h_k}],v_i)\\
(4)\quad \mathrm{d}\omega_3(v_i,h_j,\overline{v_k})&=-h_j\omega_3(v_i,\overline{v_k})-\omega_3([h_j,\overline{v_k}],v_i)\\
(5)\quad \mathrm{d}\omega_3(v_i,v_j,\overline{v_k})&=v_i\omega_3(v_j,\overline{v_k})-v_j\omega_3(v_i,\overline{v_k})+\omega_3([v_i,\overline{v_k}],v_j)-\omega_3([v_j,\overline{v_k}],v_i)\\
(6')\quad\mathrm{d}\Omega(v_i,h_j,h_k)&=-2\mathrm{i}(h_j\omega_3(v_i,\overline{v_k})-h_k\omega_3(v_i,\overline{v_j}))\\
(7')\quad \mathrm{d}\Omega(v_i,v_j,h_k)&=2\mathrm{i}(v_i\omega_3(v_j,\overline{v_k})-v_j\omega_3(v_i,\overline{v_k}))\\
(8')\quad\mathrm{d}\Omega(h_i,h_j,\overline{h_k})&=-2\mathrm{i}(\omega_3(\overline{v_j},[\overline{v_i},v_k])-\omega_3(\overline{v_i},[\overline{v_j},v_k]))\\
(9')\quad\mathrm{d}\Omega(v_i,h_j,\overline{h_k})&=-2\mathrm{i}(\overline{h_k}\omega_3(\overline{v_j},v_i)+\omega_3(\overline{v_j},[v_i,\overline{h_k}])\\
(10')\quad\mathrm{d}\Omega(h_i,h_j,\overline{v_k})&=-2\mathrm{i}(\omega_3(\overline{v_j},[h_i,\overline{v_k}])-\omega_3(\overline{v_i},[h_j,\overline{v_k}]))\\
(11')\quad\mathrm{d}\Omega(h_i,v_j,\overline{v_k})&=2\mathrm{i}(\overline{v_k}\omega_3(\overline{v_i},v_j)+\omega_3(\overline{v_i},[v_j,\overline{v_k}]))\\
\end{split}
\end{equation}
From the above, together with the flatness condition $[v_i,\overline{h_j}]=[v_j,\overline{h_i}]$ from Lemma \ref{flatlemma} we see that setting the equations to $0$ for all $i,j,k=1,...,\text{dim}_{\mathbb{C}}(M)$, gives the following implications among them (a bar over a number means the conjugate equation)
\begin{equation}
\begin{split}
    &(1')\iff (6'), \quad (7')\text{\;\;and\;\;}\overline{(8')}\implies (5),\quad (4)\iff \overline{(9')}, \quad (2')+\overline{(8')}\implies \overline{(11')}\\
    &(3)\text{\;\;and\;\;}([v_i,\overline{h_j}]=[v_j,\overline{h_i}])\implies \overline{(10')}, \quad (2')\text{\;\;and\;\;}(7')\implies \overline{(8')}, \quad (4)\text{\;\;and\;\;}([v_i,\overline{h_j}]=[v_j,\overline{h_i}])\implies (1')\\
    &\\
\end{split}
\end{equation}
Hence, we can reduce $\mathrm{d}\omega_3=0$ and $\mathrm{d}\Omega=0$ to just checking that $(2')$, $(3)$, $(4)$ and $(7')$ are equal to $0$.  Namely, we must check that 
\begin{equation}
    \begin{split}
v_i\omega_3(v_j,\overline{v_k})&=v_j\omega_3(v_i,\overline{v_k})\\
h_j\omega_3(v_i,\overline{v_k})&=\omega_3(v_i,[h_j,\overline{v_k}])\\
\omega_3([v_i,\overline{h_k}],v_j)&=\omega_3([v_j,\overline{h_k}],v_i)\\
v_i\omega_3(v_j,\overline{v_k})&=\omega_3(v_i,[v_j,\overline{v_k}])\,.\\
    \end{split}
\end{equation}
The result of the lemma then follows.

\section{Relation to the instanton generating function}\label{instgenapp}

In the work of \cite{heavenlyJoyce} a certain integral formula is given for the Pleba\'nski potential $\widetilde{J}$\footnote{In their notation it is denoted $W$.} associated to a variation of BPS structures $(\mathcal{M},\Gamma,Z,\Omega)$ \cite[Equation (1.5)]{heavenlyJoyce}. This formula is in turn related to the so-called instanton generating function $\mathcal{G}$ studied by the same authors in \cite[Equation (3.22)]{BlackholeAP},  in the context of instanton corrections in Calabi-Yau compactifications of type IIA/B string theory. We further remark that, even though the BPS indices $\Omega(\gamma)$ jump across the walls of marginal stability, the functions $\widetilde{J}$ and $\mathcal{G}$ are smooth across the walls \cite[Appendix C]{BlackholeAP}. \\

In this appendix we show that the function $J$ from \eqref{Juncoupled}, specifying the special Joyce structure associated to an uncoupled variation of BPS structures, admits an integral formula similar to the one in \cite[Equation (1.5)]{heavenlyJoyce}\footnote{The author would like to thank S. Alexandrov and B. Pioline for bringing this fact to his attention.}. In this case a second summand like in \cite[Equation (1.5)]{heavenlyJoyce} is not present due to the BPS structure being uncoupled.

\begin{proposition}
     Consider an uncoupled variation of BPS structures $(M,\Gamma,Z,\Omega)$ and the corresponding function $J$ from \eqref{Juncoupled}. Then $J$ admits the representation
    \begin{equation}
        J=\frac{1}{4\pi \mathrm{i}}\sum_{\gamma}\Omega(\gamma)\int_{l_{\gamma}}\frac{\mathrm{d}\zeta}{\zeta}\mathrm{Li}_2(\mathcal{X}_{\gamma}(\zeta))\,
    \end{equation}
    where 
    \begin{equation}\label{xsm}
        \mathcal{X}_{\gamma}(\zeta)=\exp\left(\pi\frac{Z_{\gamma}}{\zeta}+\mathrm{i}\varphi_{\gamma}+\pi\zeta\overline{Z}_{\gamma}\right),\, \quad \; l_{\gamma}=\{\zeta\in \mathbb{C} \;\; | \;\; Z_{\gamma}/\zeta\in \mathbb{R}_{<0}\}\,,
    \end{equation}
    and $\mathrm{Li}_2(z)$ denotes the dilogarithm function. 
\end{proposition}
\begin{proof}
    For $\gamma$ with $\Omega(\gamma)\neq 0$, we first parametrize $l_{\gamma}$ by $\zeta=-s\cdot Z_{\gamma}/|Z_{\gamma}|$ for $s>0$.  Recall that by the support property in Definition \ref{varBPSdef} we must have $Z_{\gamma}\neq 0$ whenever $\Omega(\gamma)\neq 0$, so this parametrization makes sense. Now note  that  
    \begin{equation}
        |\mathcal{X}_{\gamma}(-s\cdot Z_{\gamma}/|Z_{\gamma}|)|=|\exp (-s^{-1}\pi|Z_{\gamma}|+\mathrm{i}\varphi_{\gamma} - s\pi|Z_{\gamma}|)|<1\,,
    \end{equation}
    so we can use the series expansion of $\mathrm{Li}_2(z)$
    \begin{equation}
        \mathrm{Li}_2(z)=\sum_{n>0}\frac{z^n}{n^2}, \quad |z|<1\,,
    \end{equation}
    to write
    \begin{equation}
        \int_{l_{\gamma}}\frac{\mathrm{d}\zeta}{\zeta}\mathrm{Li}_2(\mathcal{X}_{\gamma}(\zeta))=\sum_{n>0}\frac{e^{\mathrm{i}n\varphi_{\gamma}}}{n^2}\int_{0}^{\infty}\frac{\mathrm{d}s}{s}\exp (-s^{-1}n\pi|Z_{\gamma}| - sn\pi|Z_{\gamma}|)\,,
    \end{equation}
    where we have used the Fubini-Tonelli theorem to exchange the sum with the integral. Now note that by using the integral representation of the Bessel function $K_{\nu}(x)$ given by
    \begin{equation}
        K_{\nu}(x)=\int_{0}^{\infty}\mathrm{d}t \; \exp(-x\cosh(t))\cdot \cosh(\nu t), \quad x>0\,,
    \end{equation}
    one finds by the substitution $s=e^t$ that
    \begin{equation}
    \begin{split}
        \int_{0}^{\infty}\frac{\mathrm{d}s}{s}\exp (-s^{-1}n\pi|Z_{\gamma}| - sn\pi|Z_{\gamma}|)&=\int_{-\infty}^{\infty}\mathrm{d}t \exp(-2n\pi |Z_{\gamma}|\cosh(t))\\
        &=2\int_{0}^{\infty}\mathrm{d}t \exp(-2n\pi |Z_{\gamma}|\cosh(t))\\
        &=2K_0(2\pi n |Z_{\gamma}|)\,.\\
    \end{split}
    \end{equation}
    Hence, it follows that
    \begin{equation}
        J=\frac{1}{2\pi \mathrm{i}}\sum_{\gamma}\Omega(\gamma)\sum_{n>0}\frac{e^{\mathrm{i}n\varphi_{\gamma}}}{n^2}K_0(2\pi n |Z_{\gamma}|)=\frac{1}{4\pi \mathrm{i}}\sum_{\gamma}\Omega(\gamma)\int_{l_{\gamma}}\frac{\mathrm{d}\zeta}{\zeta}\mathrm{Li}_2(\mathcal{X}_{\gamma}(\zeta))\,.
    \end{equation}
\end{proof}

The above proposition suggests that a function similar to \cite[Equation (1.5)]{heavenlyJoyce}, or some other related function built out of it, might give a solution to \eqref{Plebanski-add} and \eqref{Plebanski-like}, and hence define a special Joyce structure in the case of coupled variations of BPS structures.  In this case $\mathcal{X}_{\gamma}$ from \eqref{xsm} should be replaced with a solution to the TBA integral equations of \cite[Equation (5.13)]{GMN}, and the integral kernel from the second term in \cite[Equation (1.5)]{heavenlyJoyce} should be replaced with the kernel relevant to \cite{GMN}. Whether or not this is possible will be left for future work. 

\end{appendix}

\bibliography{References}
\bibliographystyle{alpha}
\end{document}